\documentclass[11pt]{amsart}
\usepackage[foot]{amsaddr}
\usepackage[english]{babel}
\usepackage[
pdftex,hyperfootnotes]
{hyperref}
\usepackage[usenames,dvipsnames,svgnames,table,x11names]{xcolor}
\usepackage{pgfplots}
\usepackage{tikz}
\usepackage{tikz-3dplot}
\usetikzlibrary{calc,shadows,shapes.callouts,shapes.geometric,shapes.misc,positioning,patterns,%
	decorations.pathmorphing,decorations.markings,decorations.fractals,decorations.pathreplacing,shadings,fadings}
\pgfplotsset{compat=1.14} 
\usepackage[utf8]{inputenc}
\usepackage{nicematrix,booktabs}

\usepackage{comment}
\usepackage[textwidth=17.5cm,textheight=22.275cm,
height=
24.275cm,
width=18cm]{geometry}
\usepackage{rotating,multirow,pdflscape}
\usepackage{amssymb,latexsym,amsmath,amsthm}
\usepackage{mathrsfs}
\usepackage{mathtools}
\usepackage{bigints}
\usepackage{tensor}

\mathtoolsset{centercolon}
\usepackage[toc,page]{appendix}

\usepackage{tocvsec2}

\usepackage{Baskervaldx}
\newtheorem{coro}{{Corollary}}

\newtheorem{theorem}{Theorem}
\newtheorem{pro}{ Proposition }
\newtheorem{lemma}{Lemma}

\newtheorem{rem}{Remark}

\newcommand{\ii}{\operatorname{i}}

\renewcommand{\d}{\operatorname{d}}
\newcommand{\Exp}[1]{\operatorname{e}^{#1}}

\newcommand{\diag}{\operatorname{diag}}

\newcommand{\C}{\mathbb{C}}

\newcommand{\N}{\mathbb{N}}
\newcommand{\R}{\mathbb{R}}

\makeatletter
\def\@settitle{\begin{center}%
		\baselineskip14\p@\relax
		\bfseries
		\uppercasenonmath\@title
		\@title
		\ifx\@subtitle\@empty\else
		\\[1ex]\uppercasenonmath\@subtitle
		\footnotesize\mdseries\@subtitle
		\fi
	\end{center}%
}
\def\subtitle#1{\gdef\@subtitle{#1}}
\def\@subtitle{}
\makeatother

\begin{document}
\title[ Generalized Charlier, Meixner and Gauss hypergeometric orthogonal polynomials]{Laguerre--Freud Equations for  the Generalized Charlier, \\Generalized Meixner and Gauss Hypergeometric \\ Orthogonal Polynomials}
%
\author[I Fernández-Irisarri]{Itsaso Fernández-Irisarri$^1$}
\email{$^1$itsasofe@ucm.es}

\author[M Mañas]{Manuel Mañas$^2$}
\email{$^2$manuel.manas@ucm.es}

\thanks{$^2$MM acknowledges Spanish ``Agencia Estatal de Investigación'' research projects [PGC2018-096504-B-C33], \emph{Ortogonalidad y Aproximación: Teoría y Aplicaciones en Física Matemática} and [PI0D2021- 122154NB-I0], \emph{Ortogonalidad y Aproximación con Aplicaciones en Machine Learning y Teoría de la Probabilidad}..}

\begin{abstract}
The Cholesky factorization of the moment matrix is considered for the generalized Charlier, generalized Meixner, and Gauss hypergeometric  discrete orthogonal polynomials. 

For the generalized Charlier, we present an alternative derivation of the Laguerre-Freud relations found by Smet and Van Assche. Third-order and second-order nonlinear ordinary differential equations are found for the recursion coefficient $\gamma_n$, that happen to be forms of the Painlevé $\text{deg-P}_{\text V}$ in disguise. Laguerre-Freud relations are also found for the generalized Meixner case, which are compared with those of Smet and Van Assche.

Finally, the  Gauss hypergeometric discrete orthogonal polynomials, also known as generalized Hahn of type I,  are also studied. Laguerre-Freud equations are found, and the differences with the equations found by Dominici and by Filipuk and Van Assche are provided.

\end{abstract}


\maketitle

\allowdisplaybreaks
\section{Introduction}

Discrete orthogonal polynomials are a distinguished and well-established part of the theory of orthogonal polynomials, and classical monographs have been devoted to this type of orthogonal polynomials. For example, the classical case is discussed in detail in \cite{NSU}, and the Riemann--Hilbert problem was considered to study asymptotics and applications in \cite{baik}. For an authoritative discussion of the subject, see \cite{Ismail, Ismail2, Beals_Wong, walter}. Semiclassical discrete orthogonal polynomials, where a discrete Pearson equation is fulfilled by the weight, have been extensively treated in the literature, see \cite{diego_paco, diego_paco1, diego, diego1}, and references therein for a comprehensive account. For some specific types of weights of generalized Charlier and Meixner types, the corresponding Freud--Laguerre type equations for the coefficients of the three-term recurrence have been studied, see for example \cite{clarkson, filipuk_vanassche0, filipuk_vanassche1, filipuk_vanassche2, smet_vanassche}.

In \cite{Manas_Fernandez-Irrisarri}, we used the Cholesky factorization of the moment matrix to study discrete orthogonal polynomials $\{P_n(x)\}_{n=0}^\infty$ on the homogeneous lattice. We considered semiclassical discrete orthogonal polynomials with weights subject to a discrete Pearson equation and, consequently, with moments constructed in terms of generalized hypergeometric functions. We introduced a banded semi-infinite matrix $\Psi$, named the Laguerre--Freud structure matrix, that models the shifts by $\pm 1$ in the independent variable of the sequence of orthogonal polynomials $\{P_n(x)\}_{n=0}^\infty$. We also showed in that paper that the contiguous relations for the generalized hypergeometric functions translate into symmetries for the corresponding moment matrix, and that the 3D Nijhoff--Capel discrete Toda lattice \cite{nijhoff, Hietarinta} describes the corresponding contiguous shifts for the squared norms of the orthogonal polynomials. In \cite{Manas}, we provided an interpretation for the contiguous transformations for the generalized hypergeometric functions in terms of simple Christoffel and Geronimus transformations. Using the Geronimus--Uvarov perturbations, we obtained determinantal expressions for the shifted orthogonal polynomials. Moreover, in \cite{Manas_Fernandez-Irrisarri2}, we explored three hypergeometric families and their associated Laguerre--Freud equations.

In this paper, we consider the generalized Charlier, generalized Meixner, and Gauss hypergeometric (also known as generalized Hahn of type I in \cite{diego_paco,diego_paco1}) discrete orthogonal polynomials, analyze the Laguerre--Freud structure matrix $\Psi$, and derive from its banded structure and its compatibility with the Toda equation and the Jacobi matrix a number of nonlinear equations for the coefficients ${\beta_n, \gamma_n}$ of the three-term recursion relations $zP_n(z) = P_{n+1}(z) + \beta_n P_n(z) + \gamma_n P_{n-1}(z)$ satisfied by the orthogonal polynomial sequence. These nonlinear recurrences for the recursion coefficients are of the type
$	\gamma_{n+1} =F_1 (n,\gamma_n,\gamma_{n-1},\dots,\beta_n,\beta_{n-1}\dots)$ and $\beta_{n+1 }= F_2 (n,\gamma_{n+1},\gamma_n,\dots,\beta_n,\beta_{n-1},\dots)$,
for some functions $F_1, F_2$. Magnus \cite{magnus, magnus1, magnus2, magnus3} named this type of relations, attending to \cite{laguerre, freud}, as Laguerre--Freud relations. There are several papers describing Laguerre--Freud relations for the generalized Charlier, generalized Meixner, and Gauss hypergeometric cases, see \cite{smet_vanassche, clarkson, filipuk_vanassche0, filipuk_vanassche1, filipuk_vanassche2, diego}.

In all cases, whether it is generalized Charlier, generalized Meixner, or Gauss hypergeometric, we observe that the $\tau$-function, defined as the Wronskian of the modified Bessel, Kummer, and Gauss hypergeometric functions, provides a solution for each system of nonlinear equations of Laguerre--Freud type for the recursion coefficients.

The layout of the paper is as follows. After this introduction, Section \ref{S:Charlier} is devoted to the generalized Charlier orthogonal polynomials. In particular, the tridiagonal Laguerre--Freud structure matrix is given in Theorem \ref{teo:generalized Charlier}, and in Proposition \ref{pro:Charlier Laguerre-Freud}, we find a set of Laguerre--Freud equations that we prove to be equivalent to those of \cite{smet_vanassche}. Here we derive third-order and second-order ODEs for the coefficient $\gamma_n$, see Theorems \ref{teo:third order} and \ref{teo:second order}, respectively. As kindly informed by Peter Clarkson, these equations are the Painlevé equation $\text{deg-P}_{\text V}$ in disguise \cite{clarkson2}. Then, in Section \ref{S:Meixner}, we treat the generalized Meixner orthogonal polynomials, and in Theorem \ref{teo:generalized Meixner}, we find the tetradiagonal Laguerre--Freud structure matrix, giving (after some simplifications) the corresponding Laguerre--Freud relations in Theorem \ref{teo:Laguerre-Freud generalized Meixner}. A comparative discussion with the Laguerre--Freud equations given by Smet \& Van Assche \cite{smet_vanassche} is also performed. Finally, in Section \ref{S:Hahn}, we discuss the Gauss hypergeometric discrete orthogonal polynomials, see \cite{diego, diego_paco}, giving the pentadiagonal Laguerre--Freud structure matrix in Theorem \ref{teo:Hahn}, and in Theorem \ref{teo:compatibilityI_Hahn}, we discuss the corresponding Laguerre--Freud relations and its comparison with the results of Dominici \cite{diego} and Filipuk \& Van Assche \cite{filipuk_vanassche2} is given, see also \cite{Dzhamay_Filipuk_Stokes}.

Now, to complete this Introduction, we give a brief \emph{resume} of the basic facts regarding discrete orthogonal polynomials and then a briefing of the relevant facts of \cite{Manas_Fernandez-Irrisarri}.

\subsection{Basics on orthogonal polynomials}Given a linear functional $\rho_z\in\C^*[z]$,  the corresponding   moment  matrix is
 \begin{align*}
 G&=(G_{n,m}), &G_{n,m}&=\rho_{n+m}, &\rho_n&= \big\langle\rho_z,z^n\big\rangle, & n,m\in\N_0:=\{0,1,2,\dots\},
 \end{align*}
with $\rho_n$ the $n$-th moment of the linear functional $\rho_z$.
If the moment matrix is such that  all its truncations, which are Hankel matrices, $G_{i+1,j}=G_{i,j+1}$,
 \begin{align*}
 G^{[k]}=\begin{bNiceMatrix}
 G_{0,0}&\Cdots &G_{0,k-1}\\
 \Vdots & & \Vdots\\
 G_{k-1,0}&\Cdots & G_{k-1,k-1}
 \end{bNiceMatrix}=\begin{bNiceMatrix}[columns-width = 0.5cm]
 \rho_{0}&\rho_1&\rho_2&\Cdots &\rho_{k-1}\\
 \rho_1 &\rho_2 &&\Iddots& \rho_k\\
 \rho_2&&&&\Vdots\\
 \Vdots& &\Iddots&&\\[9pt]
 \rho_{k-1}&\rho_k&\Cdots &&\rho_{2k-2}
 \end{bNiceMatrix}
 \end{align*}
 are nonsingular; i.e. the Hankel determinants $\varDelta_k:=\det G^{[k]} $ do not cancel, $\varDelta_k\neq 0 $, $k\in\N_0$. If this is the case, we have monic polynomials
 \begin{align}\label{eq:polynomials}
 P_n(z)&=z^n+p^1_n z^{n-1}+\dots+p_n^n, & n&\in\N_0,
 \end{align}
 with $p^1_0=0$, fulfilling the orthogonality relations 
 \begin{align*}
  \big\langle \rho, P_n(z)z^k\big\rangle &=0, & k&\in\{0,\dots,n-1\},&
   \big\langle \rho, P_n(z)z^n\big\rangle &=H_n\neq 0,
 \end{align*}
 and $\{P_n(z)\}_{n\in\N_0}$ is a sequence of  orthogonal polynomials, i.e., 
$ \big\langle\rho,P_n(z)P_m(z)\big\rangle=\delta_{n,m}H_n$ for $n,m\in\N_0$.
The   symmetric bilinear form $ \langle F, G\rangle_\rho:=\langle \rho, FG\rangle$,
is  such that the moment matrix is the Gram matrix of this bilinear form and
$\langle P_n, P_m\rangle_\rho:=\delta_{n,m} H_n$.

Introducing 
$\chi(z):=\left(\begin{NiceMatrix}
1&z&z^2&\Cdots
\end{NiceMatrix}\right)^\top$
the moment  matrix is  $G=\left\langle\rho, \chi\chi^\top\right\rangle$,
and $\chi$ is an eigenvector of the \emph{shift matrix}, $\Lambda \chi=x\chi$, where
\begin{align*}
\Lambda:=\begin{bNiceMatrix}[columns-width = auto]
	0 & 1 & 0 &\Cdots[shorten-end=7pt]&\\
	0&0 &1&\Ddots[shorten-end=10pt]&\\
	\Vdots[shorten-end=7pt]& \Ddots[shorten-end=-12pt]&\Ddots[shorten-end=10pt] &\Ddots[shorten-end=10pt]&\\
	&&&&\\[-2pt]
	&&&\textcolor{white}h&
	\end{bNiceMatrix}.
\end{align*}
Hence
$\Lambda G=G\Lambda^\top$, and the moment matrix is a Hankel matrix.

As the moment matrix symmetric its Borel--Gauss factorization is a Cholesky factorization
\begin{align}\label{eq:Cholesky}
G=S^{-1}HS^{-\top},
\end{align}
where $S$ is a lower unitriangular matrix that can be written as 
\begin{align*}
	S=\begin{bNiceMatrix}[columns-width = 20pt]
		1 & 0 &0&\Cdots[shorten-end=7pt] \\
		S_{1,0 } &  1&0&\Ddots[shorten-end=7pt]\\
		S_{2,0} & S_{2,1} &1& \Ddots[shorten-end=7pt] \\
		\Vdots[shorten-end=-6pt]  & \Ddots[shorten-end=-3pt] & \Ddots[shorten-end=-7pt]& \Ddots[shorten-end=10pt]\\[8pt]
	\end{bNiceMatrix},
	\end{align*}
	and $H=\diag(H_0,H_1,\dots)$ is a  diagonal matrix, with $H_k\neq 0$, for $k\in\N_0$.
	The Cholesky  factorization does hold whenever the principal minors of the moment matrix; i.e., the Hankel determinants $\varDelta_k$,  do not cancel.
	
	The components $P_n(z)$ of
	\begin{align}\label{eq:PS}
	P(z):=S\chi(z),
	\end{align}
	are the monic orthogonal polynomials of the functional $\rho$.

\begin{pro}\label{pro:Hankel}
	We have the determinantal expressions 
	\begin{align*}
		H_{n}&=\frac{\varDelta_{n+1}}{\varDelta_n},&
	p^1_n&=-\frac{\tilde \varDelta_n}{\varDelta_n},
	\end{align*}
with the Hankel determinants given by
\begin{align*}
	\varDelta_n&\coloneq \begin{vNiceMatrix}
		\rho_{0}&\Cdots & &\rho_{n-2}&\rho_{n-1}\\
		\Vdots    &                 &\Iddots& \Iddots&\Vdots\\
		&                 &                &                  &\\
		\rho_{n-2}&        \rho_{n-1}     &     \Cdots           &&\rho_{2n-3}\\[8pt]
		\rho_{n-1}& \Cdots&              &\rho_{2n-3}&\rho_{2n-2}
	\end{vNiceMatrix}, 
	&	\tilde \varDelta_n&\coloneq \begin{vNiceMatrix}
		\rho_{0}&\Cdots& &\rho_{n-2}&\rho_{n-1}\\
		\Vdots & &\Iddots& \Iddots&\Vdots\\
		& &\Iddots&\\
		\rho_{n-2}& \rho_{n-1}&\Cdots &&\rho_{2n-3}\\[4pt]
		\rho_{n}& \Cdots& &\rho_{2n-2}&\rho_{2n-1}
	\end{vNiceMatrix}.
\end{align*}
\end{pro}

We introduce the lower Hessenberg semi-infinite matrix
\begin{align}\label{eq:Jacobi}
J=S\Lambda S^{-1}
\end{align}
that has the vector $P(z)$ as eigenvector with eigenvalue $z$,
$JP(z)=zP(z)$.
The Hankel condition  $\Lambda G=G\Lambda^\top$ and the Cholesky factorization gives
\begin{align}\label{eq:symmetry_J}
J H=(JH)^\top =HJ^\top.
\end{align}
As  the Hessenberg matrix $JH$ is symmetric the Jacobi matrix $J$  is tridiagonal. The Jacobi matrix $J$ given in \eqref{eq:Jacobi} reads
\begin{align*}
J=\begin{bNiceMatrix}[columns-width = auto]
	\beta_0 & 1& 0&\Cdots[shorten-end=8pt]&\\
	\gamma_1 &\beta_1 & 1&\Ddots[shorten-end=10pt]&\\
	0 &\gamma_2 &\beta_2 & \Ddots[shorten-end=10pt]&\\
	\Vdots[shorten-start=5pt,shorten-end=10pt]&\Ddots[shorten-end=-8pt]& \Ddots[shorten-end=-8pt] &\Ddots[shorten-end=-3pt]&
\end{bNiceMatrix}.
\end{align*}
The eigenvalue equation $JP=zP$ is a three term recursion relation 
$zP_n(z)=P_{n+1}(z)+\beta_n P_n(z)+\gamma_n P_{n-1}(z)$,
that with the  initial conditions $P_{-1}=0$ and $P_0=1$ completely determines   the sequence of orthogonal polynomials $\{P_n(z)\}_{n=0}^\infty$ in terms of the recursion coefficients $\beta_n,\gamma_n$.
	The recursion coefficients, in terms of the Hankel determinants, are given by
	\begin{align}\label{eq:equations0}
	\beta_n&=p_n^1-p_{n+1}^1=-\frac{\tilde \varDelta_n}{\varDelta_n}+\frac{\tilde \varDelta_{n+1}}{\varDelta_{n+1}},&  \gamma_{n+1}&=\frac{H_{n+1}}{H_{n}}=\frac{\varDelta_{n+1}\varDelta_{n-1}}{\varDelta_n^2},& n\in\N_0.
	\end{align}

For future use we introduce the following diagonal matrices
$ \gamma:=\diag (\gamma_1,\gamma_2,\dots)$ and $\beta:=\diag(\beta_0 ,\beta_{1},\dots)$
and
$J_-:=\Lambda ^\top \gamma$ and $ J_+:=\beta+\Lambda$,
so that we have the splitting
$J=\Lambda^\top\gamma+\beta+\Lambda=J_-+J_+$.
In general, given any semi-infinite matrix $A$, we will write $A=A_-+A_+$, where $A_-$ is a strictly lower triangular matrix and $A_+$ an upper triangular matrix. Moreover, $A_0$ will denote the diagonal part of  $A$.

The lower Pascal matrix  is defined by
\begin{align}\label{eq:Pascal_matrix}
B&=(B_{n,m}), & B_{n,m}&:= \begin{cases}
\displaystyle \binom{n}{m}, & n\geq m,\\
0, &n<m,
\end{cases}
\end{align}
so that
\begin{align}\label{eq:Pascal}
\chi(z+1)=B\chi(z).	
\end{align}
Moreover,
\begin{align*}
B^{-1}&=(\tilde B_{n,m}), & \tilde B_{n,m}&:= \begin{cases}
(-1)^{n+m}\displaystyle \binom{n}{m}, & n\geq m,\\
0, &n<m,
\end{cases}
\end{align*}
and
$\chi(z-1)=B^{-1}\chi(z)$.	
The lower Pascal matrix and its inverse are explicitly given by
\begin{align*}
	\hspace*{-1cm}	B&=\begin{bNiceMatrix}[columns-width =auto]
		1&0&\Cdots[shorten-end=7pt]&&&&&\\
		1&1&0&\Cdots[shorten-end=7pt]&&&&\\
		1&2&1&0&\Cdots[shorten-end=7pt]&&&\\		
		1& 3 & 3&1 & 0&\Cdots[shorten-end=7pt]&&\\
		1&4 & 6 & 4 & 1&0&\Cdots[shorten-end=7pt]&\\
		1& 5 & 10 &10 &5&1&0&\Cdots[shorten-end=7pt]\\
		\Vdots[shorten-end=7pt] &\Vdots[shorten-end=7pt] & \Vdots[shorten-end=7pt]&\Vdots[shorten-end=7pt] & \Vdots[shorten-end=7pt]& \Vdots[shorten-end=7pt]&\Vdots[shorten-end=7pt]&\\\\
	\end{bNiceMatrix},&
	B^{-1}&=\begin{bNiceMatrix}[r]
		1&0&\Cdots[shorten-end=7pt]&&&&\\
		-1&1&0&\Cdots[shorten-end=7pt]&&&\\
		1&-2&1&0&\Cdots[shorten-end=7pt]&&\\		
		-1& 3 & -3&1 & 0&\Cdots[shorten-end=7pt]&\\
		1&-4 & 6 & -4 & 1&0&\Cdots[shorten-end=7pt]\\
		-1& 5 & -10 &10 &-5&1&0&\Cdots[shorten-end=7pt]\\
		\Vdots[shorten-end=7pt] &\Vdots[shorten-end=7pt] & \Vdots[shorten-end=7pt]& \Vdots[shorten-end=7pt]&\Vdots[shorten-end=7pt] &\Vdots[shorten-end=7pt]& \Vdots[shorten-end=7pt]&
	\end{bNiceMatrix},
\end{align*}
in terms of which we introduce the    {dressed Pascal matrices,}
$\Pi:=SBS^{-1}$ and  $ \Pi^{-1}:=SB^{-1}S^{-1}$, which are connection matrices; i.e.,
\begin{align}\label{eq:PascalP}
P(z+1)&=\Pi P(z), & P(z-1)&=\Pi^{-1}P(z).
\end{align}
The lower Pascal matrix can be expressed in terms of its subdiagonal structure as follows
\begin{align*}
B^{\pm 1}&=I\pm\Lambda^\top D+\big(\Lambda^\top\big)^2D^{[2]}\pm\big(\Lambda^\top\big)^3D^{[3]}+\cdots,
\end{align*}
 where 
$ D=\diag(1,2,3,\dots)$ and  $D^{[k]}=\frac{1}{k}\diag\big(k^{(k)}, (k+1)^{(k)},(k+2)^{(k)}\cdots\big)$, 
 in terms of the falling factorials
$ x^{(k)}=x(x-1)(x-2)\cdots (x-k+1)$.
That is,  
\begin{align*}
D^{[k]}_n&=\frac{(n+k)\cdots (n+1)}{k}, & k&\in\N, & n&\in\N_0.
\end{align*}

The lower unitriangular factor can be also written  in terms of its subdiagonals
$S=I+\Lambda^\top S^{[1]}+\big(\Lambda^\top\big)^2S^{[2]}+\cdots$
with 
$S^{[k]}=\diag \big(S^{[k]}_0, S^{[k]}_1,\dots\big)$.  From \eqref{eq:PS} is clear the following connection  between these subdiagonals entries and the coefficients of the orthogonal polynomials given in \eqref{eq:polynomials}
\begin{align}\label{eq:Sp}
	S^{[k]}_n=p^k_{n+k}.
\end{align} 

We will use the \emph{shift operators} $T_\pm$ acting over the diagonal matrices as follows
\begin{align*}
T_-\diag(a_0,a_1,\dots)&:=\diag (a_1, a_2,\dots),&
T_+\diag(a_0,a_1,\dots)&:=\diag(0,a_0,a_1,\dots).
\end{align*}
These shift operators have the following important properties, for any diagonal matrix $A=\diag(A_0,A_1,\dots)$
\begin{align}\label{eq:ladder_lambda}
	\Lambda A&=(T_-A)\Lambda,&   A	\Lambda &=\Lambda (T_+A), 	& A \Lambda^\top  &=\Lambda^\top(T_-A), &\Lambda^\top  A &= (T_+A)\Lambda ^\top.
\end{align}

\begin{rem}
	Notice that the standard notation, see \cite{NSU}, for the differences of a sequence $\{f_n\}_{n\in\N_0}$,
\begin{align*}
\Delta f_n&:=f_{n+1}-f_n, 
& n&\in\N_0,
\\ 
\nabla f_n&=f_n-f_{n-1}, &n&\in\N,
\end{align*}
and $\nabla f_0= f_0$, connects with the shift operators  by means of
\begin{align*}
T_-&=I+\Delta , & T_+&=I-\nabla.
\end{align*} 
\end{rem}

In terms of these shift operators we find
\begin{align}\label{eq:theDs}
	2D^{[2]}&=(T_-D)D, & 3D^{[3]}&=(T_-^2D)(T_-D)D=2(T_-D^{[2]})D=2D^{[2]}(T_-^2 D).
\end{align}

\begin{pro}\label{pro:Sinv}
The inverse matrix $S^{-1}$ of the matrix $S$  expands as follows
\begin{align*}
S^{-1}&=I+\Lambda^\top S^{[-1]}+\big(\Lambda^\top\big)^2 S^{[-2]}+\cdots.
\end{align*}
The subdiagonals $S^{[-k]}$ are  given  in terms of  the subdiagonals  of $S$. In particular, 
\begin{align*}
 S^{[-1]}&=-S^{[1]}, \\S^{[-2]}&=-S^{[2]}+(T_-S^{[1]})S^{[1]},\\
S^{[-3]}&=-S^{[3]}+(T_-S^{[2]})S^{[1]}
+(T_-^2S^{[1]})S^{[2]}
-(T_-^2 S^{[1]})
(T_-S^{[1]})S^{[1]}.
\end{align*}
\end{pro}

\begin{rem}
	Corresponding expansions  for the dressed Pascal matrices  are
\begin{align*}
\Pi^{\pm 1}=I+\Lambda^\top \pi^{[\pm 1]}+(\Lambda^\top)^2  \pi^{[\pm 2]}+\cdots
\end{align*}
with $\pi^{[\pm n]}=\diag(\pi^{[\pm n]}_0,\pi^{[\pm n]}_1,\dots)$. 
\end{rem}

\begin{pro}[The dressed Pascal matrix coefficients]
	We have
	\begin{gather}\label{eq:pis}
	\begin{aligned}
	\pi^{[\pm 1]}_n&=\pm (n+1), &
	\pi^{[\pm 2]}_n&=\frac{(n+2)(n+1)}{2}\pm
	p^1_{n+2}(n+1)\mp (n+2) p^{1}_{n+1}\\&&&=\frac{(n+2)(n+1)}{2}\mp (n+1)\beta_{n+1}
	\mp  p^{1}_{n+1},
	\end{aligned}\\\label{eq:piss}
	\begin{multlined}[t][.9\textwidth]
	\pi^{[\pm 3]}_n=\pm\frac{(n+3)(n+2)(n+1)}{3}+\frac{(n+2)(n+1)}{2}p^1_{n+3}-
	\frac{(n+3)(n+2)}{2}p^1_{n+1}\\\pm (n+1) p^2_{n+3}\mp (n+3)p^2_{n+2}\pm
	(n+3)p^1_{n+2}p^1_{n+1}\mp(n+2)p^1_{n+3}p^1_{n+1}.\end{multlined}
	\end{gather}
Moreover, the following relations are fulfill
	\begin{gather}\label{eq:pis2}
	\begin{aligned}
		\pi^{[1]}+\pi^{[-1]}&=0, &\pi^{[2]}+\pi^{[-2]}&=2D^{[2]}, &\pi^{[3]}+\pi^{[-3]}&=2((T_-^2S^{[1]})D^{[2]}-(T_-D^{[2]})S^{[1]}).
	\end{aligned}
	\end{gather}
\end{pro}

 \subsection{Discrete orthogonal polynomials and Pearson equation}
We are interested in measures with  support  on the homogeneous lattice  $\N_0$ as follows
$\rho=\sum_{k=0}^\infty \delta(z-k) w(k)$,
with  moments given by
\begin{align}\label{eq:moments}
\rho_n=\sum_{k=0}^\infty k^n w(k),
\end{align}
and, in particular,  with $0$-th moment given by
\begin{align}\label{eq:first_moment}
\rho_0=\sum_{k=0}^\infty w(k).
\end{align}

The weights  we consider in this paper satisfy the following  \emph{discrete Pearson equation}
\begin{align}\label{eq:Pearson0}
\nabla (\sigma w)&=\tau w,
\end{align}
that is $\sigma(k) w(k)-\sigma(k-1) w(k-1)=\tau(k)w(k)$, for $k\in\{1,2,\dots\}$, 
with $\sigma(z),\tau(z)\in\R[z]$.
If we write  $\theta:=\tau-\sigma$, the previous Pearson equation reads
\begin{align}\label{eq:Pearson}
\theta(k+1)w(k+1)&=\sigma(k)w(k), &
k\in\N_0.
\end{align}
\begin{theorem}[Hypergeometric symmetries]
	Let the weight $w$ be subject to a discrete Pearson equation of the type \eqref{eq:Pearson}, where the functions $\theta,\sigma$  are  polynomials, with  $\theta(0)=0$. Then, 
\begin{enumerate}
	\item 	The  moment matrix fulfills
	\begin{align}\label{eq:Gram symmetry}
	\theta(\Lambda)G=B\sigma(\Lambda)GB^\top.
	\end{align}
\item The  Jacobi matrix satisfies 
	\begin{align}\label{eq:Jacobi symmetry}
	\Pi^{-1}	H\theta(J^\top)=\sigma(J)H\Pi^\top,
\end{align}	
and the matrices $H\theta(J^\top)$ and $\sigma(J)H$ are symmetric.
\end{enumerate}
\end{theorem}

If  $N+1:=\deg\theta(z)$ and  $M:=\deg\sigma(z)$,
and  zeros of these polynomials are  $\{-b_i+1\}_{i=1}^{N}$ and $\{-a_i\}_{i=1}^M$
we write
$\theta(z)= z(z+b_1-1)\cdots(z+b_{N}-1)$ and $\sigma(z)= \eta (z+a_1)\cdots(z+a_M)$.
According to \eqref{eq:first_moment} the $0$-th moment  
\begin{align*}
\rho_0&=\sum_{k=0}^\infty w(k)=\sum_{k=0}^\infty \frac{(a_1)_k\cdots(a_M)_k}{(b_1+1)_k\cdots(b_{N}+1)_k}\frac{\eta^k}{k!}=\tensor[_M]{F}{_{N}} (a_1,\dots,a_M;b_1,\dots,b_{N};\eta)
=
{\displaystyle \,{}_{M}F_{N}\left[{\begin{matrix}a_{1}&\cdots &a_{M}\\b_{1}&\cdots &b_{N}\end{matrix}};\eta\right].}
\end{align*}
is the generalized hypergeometric function, where we are using the two standard 
notations,
see \cite{generalized_hypegeometric_functions}.
Then,  according to
\eqref{eq:moments}, for $n\in\N$, the corresponding  higher moments  $\rho_n=\sum_{k=0}^\infty k^n w(k)$, are
\begin{align*}
\rho_n&=\vartheta_\eta^n\rho_0=\vartheta_\eta^n\Big({\displaystyle \,{}_{M}F_{N}\left[{\begin{matrix}a_{1}&\cdots &a_{M}\\b_{1}&\cdots &b_{N}\end{matrix}};\eta\right]}\Big), &\vartheta_\eta:=\eta\frac{\partial }{\partial \eta}.
\end{align*}

Given a function $f(\eta)$, we consider the Wronskian of the  covector $\delta(f)\coloneq\begin{bNiceMatrix}
	f &\vartheta f&\Cdots&\vartheta^n f
\end{bNiceMatrix}$ given by
\begin{align*}
	\mathscr W_n(f)\coloneq\begin{vNiceMatrix}[columns-width = 0.5cm]
		f &\vartheta_\eta f& \vartheta_\eta^2f&\Cdots &\vartheta_\eta^nf\\
		\vartheta_\eta f& \vartheta_\eta^2f&&\Iddots& \vartheta_\eta^{n+1}f\\
		\vartheta_\eta^2 f&&&&\Vdots\\
		\Vdots& &\Iddots&&\\
		\vartheta_\eta^nf&\vartheta_\eta^{n+1}f& \Cdots&&\vartheta_\eta^{2n}f
	\end{vNiceMatrix}.
\end{align*}
We refer to this Wronskian as the $\delta$-Wronskian of $f$.
Then, we have that the Hankel determinants $\varDelta_n=\det G^{[n]}$ determined by the truncations of the corresponding moment  matrix are $\delta$-Wronskians of generalized hypergeometric functions,
\begin{align}\label{eq:hankel_hyper}
\varDelta_n&=\tau_n, &
\tilde \varDelta_ n&=\vartheta_\eta\tau_n,
\end{align}
where
\begin{align*}
	 \tau_n&:=\mathscr W_{n}\Big({\displaystyle \,{}_{M}F_{N}\left[{\begin{matrix}a_{1}&\cdots &a_{M}\\b_{1}&\cdots &b_{N}\end{matrix}};\eta\right]}\Big).
\end{align*}
Moreover, using Proposition \ref{pro:Hankel} we get
\begin{align}\label{eq:Wp_n}
H_n&=\frac{\tau_{n+1}}{\tau_n},	&
p^1_n&=-\vartheta_\eta\log \tau_n, &n&\in\N_0.
\end{align}
The functions $\tau_k$ are the well known  tau functions \cite{Hietarinta}. In terms of these $\tau$-functions we have
	\begin{align}\label{eq:tau_recursion}
	\beta_n&=\vartheta_{\eta}\log \frac{\tau_{n+1}}{\tau_n},&  \gamma_{n+1}&=\frac{\tau_{n+1}\tau_{n-1}}{\tau_n^2},& n\in\N_0,
\end{align}
were we have used \eqref{eq:hankel_hyper}, \eqref{eq:equations0} and \eqref{eq:Wp_n} .

\begin{rem}
Let us observe that this Wronskian structure holds not only for weights that solve Pearson's equations. In fact, the only requirement is that the weight depends on a parameter $\eta$ such that $\vartheta_{\eta} w(k)=k w(k)$, where $k\in\N_0$. In such cases, the tau function is given by $\tau_n=\mathscr W_n(\rho_0)$, where $\rho_0=\sum_{k=0}^\infty  w(k)$, and Equations \eqref{eq:hankel_hyper}, \eqref{eq:Wp_n}, and \eqref{eq:tau_recursion} hold true.\end{rem}

\begin{theorem}[Laguerre--Freud structure matrix]
	Let us assume that the weight $w$  solves the discrete Pearson equation \eqref{eq:Pearson} with  $\theta,\sigma$ polynomials such that $\theta(0)=0$, $\deg\theta(z)=N+1$, $ \deg\sigma(z)=M$. 	Then, the Laguerre--Freud structure matrix
\begin{align}\label{eq:Psi}
\Psi&:=\Pi^{-1}H\theta(J^\top)=\sigma(J)H\Pi^\top=\Pi^{-1}\theta(J)H=H\sigma(J^\top)\Pi^\top\\
&=\theta(J+I)\Pi^{-1} H=H\Pi^\top\sigma(J^\top-I),\label{eq:Psi2}
\end{align}
has only  $N+M+2$ possibly nonzero  diagonals ($N+1$ superdiagonals and $M$ subdiagonals) 
\begin{align*}
\Psi=(\Lambda^\top)^M\psi^{(-M)}+\dots+\Lambda^\top \psi^{(-1)}+\psi^{(0)}+
\psi^{(1)}\Lambda+\dots+\psi^{(N+1)}\Lambda^{N+1},
\end{align*}
for some diagonal matrices $\psi^{(k)}$. In particular,  the lowest subdiagonal and highest superdiagonal  are given by
\begin{align}\label{eq:diagonals_Psi}
\left\{
\begin{aligned}
(\Lambda^\top)^M\psi^{(-M)}&=\eta(J_-)^MH,&
\psi^{(-M)}=\eta H\prod_{k=0}^{M-1}T_-^k\gamma=\eta\diag\Big(H_0\prod_{k=1}^{M}\gamma_k, H_1\prod_{k=2}^{M+1}\gamma_k,\dots\Big),\\
\psi^{(N+1)} \Lambda^{N+1}&=H(J_-^\top)^{N+1},&
\psi^{(N+1)}=H\prod_{k=0}^{N}T_-^k\gamma=\diag\Big(H_0\prod_{k=1}^{N+1}\gamma_k, H_1\prod_{k=2}^{N+2}\gamma_k,\dots\Big).
\end{aligned}
\right.
\end{align}
The vector $P(z)$ of orthogonal polynomials fulfill the following structure equations
\begin{align}\label{eq:P_shift}
\theta(z)P(z-1)&=\Psi H^{-1} P(z), &
\sigma(z)P(z+1)&=\Psi^\top H^{-1} P(z).
\end{align}
\end{theorem}

The compatibility of the recursion relation, i.e. eigenfunctions of the Jacobi matrix, and the recursion matrix leads to some interesting equations:
\begin{pro}
	The following compatibility conditions for the Laguerre--Freud  and Jacobi matrices  hold
\begin{subequations}
		\begin{align}\label{eq:compatibility_Jacobi_structure_a}
	[\Psi H^{-1},J]&=\Psi H^{-1}, \\ \label{eq:compatibility_Jacobi_structure_b}	[J, \Psi ^\top H^{-1}]&=\Psi ^\top H^{-1}.
	\end{align}
\end{subequations}
\end{pro}

\subsection{The Toda flows}

Let us define the strictly lower triangular matrix
\begin{align*}
\Phi&:=(\vartheta_{\eta} S ) S^{-1}.
\end{align*}

\begin{pro}
\begin{enumerate}
	\item 	The semi-infinite vector  $P$ fulfills
	\begin{align}\label{eq:MP}
	\vartheta_{\eta} P=\Phi  P.
	\end{align}

\item 	The Sato--Wilson equations holds
	\begin{align}\label{eq:Mscr}
	-\Phi  H+\vartheta_\eta H-H \Phi^\top&=JH.
	\end{align}
	Consequently, $\Phi=-J_-$ and $n\in\N_0$ we have 
$\vartheta_\eta \log H_n=J_{n,n}$.
\end{enumerate}
\end{pro}

 Moreover,

\begin{pro}[Toda]
The	 following  equations hold 
\begin{subequations}
		\begin{align}\label{eq:MS}
	\Phi =(\vartheta_{\eta} S)S^{-1}&=-\Lambda^\top\gamma,\\\label{eq:MS0}
	(\vartheta_\eta H) H^{-1}&=\beta.
	\end{align}
\end{subequations}
	In particular,  for $n,k-1\in\N$, we have
\begin{subequations}
		\begin{gather}\label{eq:equations}
\begin{aligned}
		\vartheta_\eta p^1_n&=-\gamma_n, & \vartheta_{\eta}p^k_{n+k}&=-\gamma_{n+k}p^{k-1}_{n+k-1}, 
\end{aligned}\\
	\vartheta_\eta \log H_n=\beta_ n.\label{eq:equationsH}
	\end{gather}
\end{subequations}
	The functions $q_n:=\log H_n$, $n\in\N$, satisfy the Toda equations
	\begin{align}\label{eq:Toda_equation}
	\vartheta_\eta^2q_n=\Exp{q_{n+1}-q_n}-\Exp{q_n-q_{n-1}}.
	\end{align}
	For $n\in\N$, we also have
$\vartheta_\eta P_{n}(z)=-\gamma_n P_{n-1}(z)$.
\end{pro}

\begin{pro}
	The following Lax equation holds
	$	\vartheta_\eta J=[J_+,J]$.
	The  recursion coefficients satisfy the following Toda system
\begin{subequations}\label{eq:Toda_system}
		\begin{align}\label{eq:Toda_system_beta}
	\vartheta_\eta\beta_n&=\gamma_{n+1}-\gamma_n,\\\label{eq:Toda_system_gamma}
	\vartheta_\eta\log\gamma_n&=\beta_{n}-\beta_{n-1},
	\end{align}
\end{subequations}
	for $n\in\N_0$ and $\beta_{-1}=0$. Consequently, we get 
	\begin{align}\label{eq:Toda_equation_gamma}
	\vartheta_\eta^2\log\gamma_n&=\gamma_{n+1}-2\gamma_n+\gamma_{n-1}.
	\end{align}
\end{pro}

For the compatibility of  \eqref{eq:PascalP} and \eqref{eq:MP}, that is for the compatibility of the systems
\begin{align*}
\begin{cases}
\begin{aligned}
P(z+1)&=\Pi P(z),\\
\vartheta_\eta (P(z))&=\Phi  P(z).
\end{aligned}
\end{cases}
\end{align*}
 we obtain
$\vartheta_\eta(\Pi )=[\Phi ,\Pi ]$.
In the general case the dressed Pascal matrix $\Pi$ is a lower unitriangular semi-infinite matrix,  that possibly has an infinite number of subdiagonals. However, for the case when the weight $w(z)=v(z)\eta^z$ satisfies the Pearson equation \eqref{eq:Pearson}, with $v$ independent of $\eta$, that is
$\theta(k+1)v(k+1)\eta=\sigma(k)v(k)$,
the situation improves as we have the banded semi-infinite matrix $\Psi$ that models the shift in the $z$ variable as in \eqref{eq:P_shift}. From the previous discrete 
Pearson equation we see that  $\sigma(z)=\eta\kappa(z)$ with $\kappa,\theta$  $\eta$-independent polynomials in $z$
\begin{align}\label{eq:Pearson_Toda}
\theta(k+1)v(k+1)=\eta\kappa(k)v(k).
\end{align}

\begin{pro}
	Let us assume   a weight $w$ satisfying the Pearson equation
	\eqref{eq:Pearson}. Then, the Laguerre--Freud structure matrix $ \Psi$ given in \eqref{eq:Psi} satisfies
\begin{subequations}
		\begin{align}\label{eq:eta_compatibility_Pearson_1a}
	\vartheta_\eta(\eta^{-1}\Psi^\top H^{-1} )&=[\Phi ,\eta^{-1}\Psi^\top H^{-1}  ], \\
	\vartheta_\eta(\Psi H^{-1} )&=[\Phi ,\Psi H^{-1} ].\label{eq:eta_compatibility_Pearson_1b}
	\end{align}
\end{subequations}
Relations \eqref{eq:eta_compatibility_Pearson_1a} and \eqref{eq:eta_compatibility_Pearson_1b} are \emph{gauge} equivalent.
\end{pro}

\section{Generalized Charlier  weights}\label{S:Charlier}

A particular simple case is when $\sigma$ does not depend upon $z$, that we name as extended Charlier weights. The generalized Charlier,   which is the main example of these extended Charlier weights,  corresponds to the choice  $\sigma(z)=\eta$ and $\theta(z)= z(z+b)$. In fact, a weight that satisfies the  Pearson equation
\begin{align*}
(k+1)(k+1+b)w(k+1)=\eta w(k)
\end{align*}
is proportional to the generalized Charlier weight 
\begin{align*}
w(z)=\frac{\Gamma(b+1)\eta^z}{\Gamma(z+b+1)\Gamma(z+1)}=\frac{1}{(b+1)_z}\frac{\eta^z}{z!}.
\end{align*}
The finiteness of the moments requires $b>-1$ and $\eta>0$, see \cite{diego_paco}. The generalized Charlier $0$-th moment, as we have discussed,  is the confluent hypergeometric limit function
$\rho_0(\eta)={}_0F_1(;b+1;\eta)$.

\begin{rem}
	Is well known the relation of this function with the Bessel functions $J_{\nu}$ and $I_{\nu}$:
\begin{align*}
\frac {({\tfrac {\eta}{2}})^{b}}{\Gamma (b+1)}\rho_{0}\Big(-\frac{\eta^2}{4}\Big)&= J_{b}(\eta),&
\frac {({\tfrac {\eta}{2}})^{b}}{\Gamma (b+1)}\rho_{0}\Big(\frac{\eta^2}{4}\Big)&= I_{b}(\eta).
\end{align*}
Here $J_{\nu}$ is the $\nu$-th Bessel function and $I_{\nu}$ the $\nu$-th modified Bessel, both are connected: $\Exp{\pm \ii \frac{\nu\pi}{2}}I_{\nu}(x)= 	J_{\nu}(\Exp{\pm \ii \frac{\nu\pi}{2}} x) $. Thus, in terms of these modified Bessel functions, we can write, see \cite{diego_paco}
\begin{align}\label{eq:Bessel_moment}
\rho_0(\eta)=\frac{\Gamma(b+1)}{\sqrt{\eta}^b}I_{b}(2\sqrt{\eta}).
\end{align}
\end{rem}

\begin{rem}
	From \eqref{eq:Bessel_moment} and \eqref{eq:tau_recursion} we get that in terms of the following  $\delta$-Wronskian  of a modified Bessel function
	\begin{align*}
	\tau_n&:=\mathscr W_{n}\left(\frac{\Gamma(b+1)}{\sqrt{\eta}^b}I_{b}(2\sqrt{\eta})\right),
	\end{align*}
	we get explicit expressions for the recursion coefficients
	\begin{align*}
	\beta_n&=\vartheta_{\eta}\log \frac{\tau_{n+1}}{\tau_n},&  \gamma_{n+1}&=\frac{\tau_{n+1}\tau_{n-1}}{\tau_n^2},& n\in\N_0.
	\end{align*}
\end{rem}

\begin{rem}
	The\emph{ non generalized} Charlier polynomials or Poisson--Charlier polynomials were discussed by Charlier  in \cite{charlier}. 
\end{rem}

\begin{pro}
	Let us consider an extended Charlier type weight $w$, i.e. 
	$\theta(k+1)w(k+1)=\eta w(k)$ with $\theta(0)=0$. Then, the semi-infinite matrix $H\theta(J^\top)$ admits the  Cholesky factorization 
	\begin{align}\label{eq:Pascal_Chirstoffel}
\theta(J)H=	H\theta(J^\top)=\Theta^{-1} H \Theta^{-\top}=\Pi H\Pi^\top,
	\end{align}
	with the dressed Pascal semi-infinite matrix $\Pi=\Theta^{-1}$ being a lower unitriangular semi-infinite matrix with only its first
	$\deg\theta$ subdiagonals different from zero.
\end{pro}
\begin{proof}
	From \eqref{eq:Psi} we have $\Pi^{-1}H\theta(J^\top)=H\Pi^\top$; i.e.
	 $H\theta(J^\top)=\Pi H\Pi^\top$ and given the uniqueness of the Cholesky factorization and its band structure we deduce the result.
\end{proof}

\begin{pro}
	For an extended Charlier type weight, with the choice  $\sigma=\eta$ and $\theta(z)=z( z^{N-1}+\cdots+\theta_1)$, we have  the following subdiagonal structure for the dressed Pascal matrix
$\Pi=\sum_{k=0}^{N}(\Lambda^\top)^k\pi^{[k]}$
with the main diagonal and first subdiagonals given by
$ \pi^{[0]}=I$ and $\pi^{[1]}=D=\diag(1,2,3,\cdots)$, 
 and the lowest subdiagonal
\begin{align*}
(\Lambda^\top)^N\pi^{[N]}&=(J_-)^{N},&	\pi^{[N]}=\prod_{k=0}^{N-1}T_-^k\gamma=\diag\Big(\prod_{k=1}^{N}\gamma_k, \prod_{k=2}^{N+1}\gamma_k,\dots\Big)=\diag\Big(\frac{H_N}{H_0}, \frac{H_{N+1}}{H_1},\dots\Big).
\end{align*}
\end{pro}
\begin{proof}
	The only new fact to prove is the explicit expression for the lowest subdiagonal, that obviously will come from the lowest diagonal of $\theta(J)$ that is $ J_-^N$.
\end{proof}

\begin{theorem}[  The generalized Charlier Laguerre--Freud  structure matrix]\label{teo:generalized Charlier}
	For a generalized Charlier weight,  i.e. $\sigma=\eta$ and $\theta= z(z+b)$,
the Laguerre--Freud   structure matrix is
\begin{align*}
\Psi&
=
\begin{bNiceMatrix}[columns-width = 1cm]
	\eta	H_0 & \eta H_0 &H_2 &0&0&\Cdots[shorten-end=10pt]\\
	0 & \eta H_1 & 2 \eta H_1& H_3  &0&\Ddots[shorten-end=10pt]\\
	0 & 0&\eta H_2&3\eta H_2& H_4&\Ddots[shorten-end=10pt]  \\
	\Vdots[shorten-end=7pt] &\Ddots[shorten-end=87pt] &\Ddots[shorten-end=77pt]& \Ddots[shorten-end=87pt]&\Ddots[shorten-end=35pt]&\Ddots[shorten-end=10pt]\\
	&&&&&\\
	&&&&&
\end{bNiceMatrix}.
\end{align*}
\end{theorem}
\begin{proof}
For the structure matrix $\Psi$ we have that, see \eqref{eq:diagonals_Psi},
$\Psi=\psi^{(0)}+\psi^{(1)}\Lambda +\psi^{(2)}\Lambda^2$.
Let us find these diagonal matrix coefficients.
In the one hand, as
\begin{align}\label{eq:Charlier_Psi_0}
\Psi=\eta H\Pi^\top=\underset{\text{main diagonal}}{ \underbrace{\eta H}}+\underset{\text{first superdiagonal}}{ \underbrace{\eta H D\Lambda}}+ \underset{\text{second superdiagonal}}{ \underbrace{\eta H\pi^{[2]} \Lambda^2}}+\cdots,
\end{align}
we get from the main diagonal that $\psi^{(0)}=\eta H$, and from the first superdiagonal  that $\psi^{(1)}=\eta HD$.
In the other hand, we observe that
\begin{align*}
\Psi&=\Pi^{-1}H\Big((J^\top)^2+bJ^\top\Big)\\&=
\begin{multlined}[t][.9\textwidth]
\big(I-\Lambda^\top D+(\Lambda^\top)^2\pi^{[-2]}+\cdots\big)H\\
\times \big((\Lambda^\top)^2+
\Lambda^\top (\beta+{T_-}\beta+bI)+\gamma+T_+\gamma+\beta^2+b\beta+ (\beta+{T_-}\beta+bI)\gamma \Lambda+({T_-}\gamma )\gamma\Lambda^2\big)
\end{multlined}
\end{align*}
so that
\begin{multline}\label{eq:Charlier_Psi_1}
\Psi=\cdots
+\underset{\text{main diagonal}}{ \underbrace{H(\gamma+T_+\gamma+\beta^2+b\beta)
-\Lambda^\top D H(\beta+{T_-}\beta+bI)\gamma \Lambda+(\Lambda^\top)^2\pi^{[-2]}H({T_-}\gamma )\gamma\Lambda^2}}\\+
\underset{\text{first superdiagonal}}{ \underbrace{H(\beta+{T_-}\beta+bI)\gamma \Lambda-\Lambda^\top DH({T_-}\gamma )\gamma\Lambda^2}}+\underset{\text{second superdiagonal}}{ \underbrace{
		H({T_-}\gamma )\gamma\Lambda^2,}}
\end{multline}
and the result is proven.
\end{proof}\enlargethispage{1cm}

\begin{pro}[Compatibility]
For a generalized Charlier weight 	the recursion coefficients fulfill
		\begin{align}\label{eq:charlier_compatibility}
n	\eta (\beta_n-\beta_{n-1}-1)&=(-\gamma_{n+1}+\gamma_{n-1})\gamma_n,
	\end{align}
for $n\in\N$ and $\gamma_0=0$.
Alternative forms of this equation are 
\begin{subequations}
	\begin{align}\label{eq:eta_compatibility_charlier_1}
	n\vartheta_\eta\Big(\frac{\eta}{\gamma_n}\Big)&=\gamma_{n+1}-\gamma_{n-1},\\
\label{eq:eta_compatibility_charlier_2}
	\vartheta_\eta\Big(\frac{\gamma_n\gamma_{n+1}}{\eta}\Big)&=
(n+1)\gamma_{n}-	n\gamma_{n+1},
\end{align}
\end{subequations}
for $n\in\N$ and $\gamma_0=0$.
We also have
	\begin{align}\label{eq:a dos pasos}
	\beta_n-\beta_{n-2}-1&=\eta\big(n\gamma_{n}^{-1}-(n-1)\gamma_{n-1}^{-1}\big), & n\geq 2.
\end{align}
\end{pro}

\begin{proof}
	Recalling $\gamma=({T_-}H)H^{-1}$ we write 
	\begin{align*}
	\Psi H^{-1}&=\eta+\eta HD\Lambda H^{-1}+
	H({T_-}\gamma)\gamma\Lambda^2 H^{-1}=\eta+\eta H({T_-}H)^{-1}D\Lambda +H({T_-}^2H)^{-1}({T_-}\gamma)\gamma\Lambda^2 \\&=
	\eta+\eta \gamma^{-1}D\Lambda +\Lambda^2\\
	&=	\begin{bNiceMatrix}[columns-width = auto]
		\eta & \dfrac{\eta}{\gamma_1} &1 & 0 &\Cdots[shorten-end=7pt]
		\\
		0	&\eta &\dfrac{2\eta}{\gamma_2} &1 &\Ddots[shorten-end=7pt]	\\
		0&0	&\eta &\dfrac{3\eta}{\gamma_3} &\Ddots[shorten-end=7pt]\\
		0&0&0&\eta&\Ddots[shorten-end=10pt]\\
		\Vdots& \Ddots[shorten-end=15pt]&\Ddots[shorten-end=15pt]&\Ddots[shorten-end=15pt]&\Ddots[shorten-end=35pt]
	\end{bNiceMatrix}.
	\end{align*}
		The compatibility equation \eqref{eq:compatibility_Jacobi_structure_b},  $	[\Psi H^{-1},J]=\Psi H^{-1}$, reads
	\begin{gather*}
	\begin{bNiceMatrix}[columns-width = auto]
		\eta & \dfrac{\eta}{\gamma_1} &1 & 0 &\Cdots[shorten-end=8pt]
		\\
		0	&\eta &\dfrac{2\eta}{\gamma_2} &1 &\Ddots[shorten-end=10pt]	\\
		0&0	&\eta &\dfrac{3\eta}{\gamma_3} &\Ddots[shorten-end=10pt]\\
		0&0&0&\eta&\Ddots[shorten-end=14pt]\\
		\Vdots[shorten-end=6pt]& \Ddots[shorten-end=20pt]&\Ddots[shorten-end=20pt]&\Ddots[shorten-end=22pt]&\Ddots[shorten-end=40pt]
	\end{bNiceMatrix}	
	=\begin{aligned}[t]
		&-	\begin{bNiceMatrix}[columns-width = auto]
			\beta_0 & 1& 0&\Cdots[shorten-end=8pt]&\\
			\gamma_1 &\beta_1 & 1&\Ddots[shorten-end=10pt]&\\
			0 &\gamma_2 &\beta_2 & \Ddots[shorten-end=10pt]&\\
			\Vdots[shorten-start=5pt,shorten-end=10pt]&\Ddots[shorten-end=-8pt]& \Ddots[shorten-end=-8pt] &\Ddots[shorten-end=-3pt]&
		\end{bNiceMatrix}		\begin{bNiceMatrix}[columns-width = auto]
			\eta & \dfrac{\eta}{\gamma_1} &1 & 0 &\Cdots[shorten-end=8pt]
			\\
			0	&\eta &\dfrac{2\eta}{\gamma_2} &1 &\Ddots[shorten-end=10pt]	\\
			0&0	&\eta &\dfrac{3\eta}{\gamma_3} &\Ddots[shorten-end=10pt]\\
			0&0&0&\eta&\Ddots[shorten-end=14pt]\\
			\Vdots[shorten-end=6pt]& \Ddots[shorten-end=20pt]&\Ddots[shorten-end=20pt]&\Ddots[shorten-end=22pt]&\Ddots[shorten-end=40pt]
		\end{bNiceMatrix}			\\&+			\begin{bNiceMatrix}[columns-width = auto]
			\eta & \dfrac{\eta}{\gamma_1} &1 & 0 &\Cdots[shorten-end=8pt]
			\\
			0	&\eta &\dfrac{2\eta}{\gamma_2} &1 &\Ddots[shorten-end=10pt]	\\
			0&0	&\eta &\dfrac{3\eta}{\gamma_3} &\Ddots[shorten-end=10pt]\\
			0&0&0&\eta&\Ddots[shorten-end=14pt]\\
			\Vdots[shorten-end=6pt]& \Ddots[shorten-end=20pt]&\Ddots[shorten-end=20pt]&\Ddots[shorten-end=22pt]&\Ddots[shorten-end=40pt]
		\end{bNiceMatrix}	\begin{bNiceMatrix}[columns-width = auto]
			\beta_0 & 1& 0&\Cdots[shorten-end=8pt]&\\
			\gamma_1 &\beta_1 & 1&\Ddots[shorten-end=10pt]&\\
			0 &\gamma_2 &\beta_2 & \Ddots[shorten-end=10pt]&\\
			\Vdots[shorten-start=5pt,shorten-end=10pt]&\Ddots[shorten-end=-8pt]& \Ddots[shorten-end=-8pt] &\Ddots[shorten-end=-3pt]&
		\end{bNiceMatrix}	
	\end{aligned}
\end{gather*}
			and, consequently, 	we get \eqref{eq:charlier_compatibility} on the first superdiagonal and from the second superdiagonal we get \eqref{eq:a dos pasos}.

Let us look to the alternative expression \eqref {eq:eta_compatibility_charlier_1}. First, using the Toda system 
\eqref{eq:Toda_system_gamma} we see that \eqref {eq:eta_compatibility_charlier_1}  is   equivalent to \eqref{eq:charlier_compatibility}. Indeed, from  \eqref{eq:Toda_system_gamma}  we get
\begin{align*}
	n\vartheta_\eta\Big(\frac{\eta}{\gamma_n}\Big)=\frac{n\eta}{\gamma_n}-\frac{n\eta\vartheta_{\eta} \gamma_n}{\gamma_n^2}=-	\frac{n\eta(\beta_n-\beta_{n-1}-1)}{\gamma_n},
\end{align*}
and the statement follows.
 An alternative proof for \eqref {eq:eta_compatibility_charlier_1}
is  obtained from the compatibility condition \eqref{eq:eta_compatibility_Pearson_1b}, i.e.,  

\begin{align*}
	\vartheta_\eta	\begin{bNiceMatrix}[columns-width = auto]
		\eta & \dfrac{\eta}{\gamma_1} &1 & 0 &\Cdots[shorten-end=8pt]
		\\
		0	&\eta &\dfrac{2\eta}{\gamma_2} &1 &\Ddots[shorten-end=10pt]	\\
		0&0	&\eta &\dfrac{3\eta}{\gamma_3} &\Ddots[shorten-end=10pt]\\
		0&0&0&\eta&\Ddots[shorten-end=14pt]\\
		\Vdots[shorten-end=6pt]& \Ddots[shorten-end=20pt]&\Ddots[shorten-end=20pt]&\Ddots[shorten-end=22pt]&\Ddots[shorten-end=40pt]
	\end{bNiceMatrix}	
	&=
	\begin{bNiceMatrix}[columns-width = 1cm]
		\eta & \gamma_2& 0&0&\Cdots \\
		0&\eta& \gamma_3-\gamma_1 &0 &\Ddots  \\
		0&0& \eta & \gamma_4-\gamma_2 &\Ddots
		\\
		\Vdots[shorten-end=5pt]&\Ddots[shorten-end=35pt]&\Ddots[shorten-end=57pt] &\Ddots[shorten-end=65pt]  & \Ddots[shorten-end=7pt]\\
		&&&&
	\end{bNiceMatrix}.
\end{align*}
Equation \eqref{eq:eta_compatibility_Pearson_1a} with 
\begin{align*}
	\Pi=\eta^{-1}\Psi^\top  H^{-1}=\begin{bNiceMatrix}[columns-width = auto	]
		1 & 0 &0&\Cdots[shorten-end=5pt]\\
		1 & 1& 0&\Ddots\\
		\frac{\gamma_2\gamma_1}{\eta}&2&1&\Ddots\\
		0& \frac{\gamma_3\gamma_2}{\eta} &3 &\Ddots\\[5pt]
		\Vdots[shorten-start=5pt,shorten-end=5pt] &\Ddots[shorten-end=20pt]&\Ddots[shorten-end=5pt]&\Ddots[shorten-end=30pt]
	\end{bNiceMatrix} 
\end{align*}
gives \eqref{eq:eta_compatibility_charlier_2}.
\end{proof}

\begin{theorem}[Third order ODE for generalized Charlier]\label{teo:third order}
	The recursion coefficient $\gamma_n$ of the generalized Charlier polynomials is subject to the following third order nonlinear ODE
\begin{align}\label{eq:ode_gamma_2}
	\vartheta_{\eta}\Big(\frac{\gamma_n}{\eta}(\vartheta_\eta^2\log\gamma_n+2\gamma_n)+n^2\frac{\eta}{\gamma_n}\Big)=2\gamma_n
\end{align}
\end{theorem}
\begin{proof}
In the one hand, Equation  \eqref{eq:eta_compatibility_charlier_2} can be written as follows
\begin{align*}
	\pm\vartheta_\eta\Big(\frac{\gamma_n\gamma_{n\pm1}}{\eta}\Big)&=
	(n\pm 1)\gamma_{n}-	n\gamma_{n\pm 1},
\end{align*}
and, consequently, we find
\begin{align*}
	\vartheta_\eta\Big(\frac{\gamma_n\gamma_{n+1}+\gamma_{n-1}\gamma_n}{\eta}\Big)&=
	(n+ 1)\gamma_{n}-	n\gamma_{n+ 1}-	(n- 1)\gamma_{n}+	n\gamma_{n-1}=2\gamma_n-n(\gamma_{n+1}-\gamma_{n-1})\\&=
2\gamma_n-n^2\vartheta_\eta\Big(\frac{\eta}{\gamma_n}\Big),
\end{align*}
where \eqref{eq:eta_compatibility_charlier_1} has been used.
In the other hand, 
\begin{align*}
	\vartheta_\eta\Big(\frac{\gamma_n\gamma_{n+1}+\gamma_{n-1}\gamma_n}{\eta}\Big)&=\vartheta_\eta\Big(
	\frac{\gamma_n}{\eta}(\gamma_{n+1}+\gamma_{n-1})\Big)
	=\vartheta_\eta\Big(	\frac{\gamma_n}{\eta}(\vartheta_\eta^2\log\gamma_n+2\gamma_n)\Big),
\end{align*}
where the Toda equation \eqref{eq:Toda_equation_gamma} for the $\gamma$'s has been used. Hence, comparing the previous equations we get Equation \eqref{eq:ode_gamma_2}.
\end{proof}

\begin{rem}Using the relations 
	\begin{align*}
		\vartheta_{\eta}^2\log\gamma_n&=\vartheta_{\eta}\Big(\frac{\eta}{\gamma_n}\Big)\frac{\d\gamma_n}{\d\eta}
		+\frac{\eta^2}{\gamma_n}\frac{\d^2\gamma_n}{\d\eta^2},&
		\vartheta_{\eta}\Big(\frac{\eta}{\gamma_n}\Big)&=\frac{\eta}{\gamma_n}-\frac{\eta^2}{\gamma_n^2}\frac{\d\gamma_n}{\d\eta}
	\end{align*}
	we write \eqref{eq:ode_gamma_2} as follows 
		\begin{align*}
			\frac{\d}{\d\eta}\big(\eta \mathcal P'_{\text{III},n}(\gamma_n)
			\big)&=2
			\frac{\gamma_{n}}{\eta}, & \mathcal P'_{\text{III},n}(\gamma_n)&\coloneq \frac{\d^2\gamma_n}{\d\eta^2}
			-\frac{1}{\gamma_n}\Big(\frac{\d\gamma_n}{\d\eta}\Big)^2+\frac{1}{\eta}\frac{\d\gamma_n}{\d\eta}+\frac{2\gamma_n^2}{\eta^2}+\frac{n^2}{\gamma_n} .
		\end{align*}
		The notation here is motivated by the Okamoto's, see \cite{Okamoto}, alternative form  $\text{P}_{\text{III}}'$ of the Painlevé III equation ($\mathcal P'_{\text{III},n}(u)=0$),  listed as   \href{https://dlmf.nist.gov/32.2}{32.2.9} in the Digital Library of Mathematical Functions (DLMF) at NIST, with the following choice of parameters  given there:  $\alpha=8$, $\gamma=\beta=0$ and $\delta=4n^2$. The Okamoto's notation in \cite{Okamoto} is $\text{P}_{\text{III}'}$. In fact, the corresponding Painlevé equation is, after suitable rescaling of dependent and independent variable,  the $\text{P}_{\text{III}'}(D7)$  equation (9) with $\beta=0$   in \cite{Okamoto2}. See Theorem 1 (iv) in \cite{Okamoto2} and the comment immediately after.
		For the connection with  $\text{P}_{\text{III}}$ see \cite{smet_vanassche,walter} and \cite{clarkson}.  In particular, Remark 4.5 in \cite{clarkson}  gives the functions $p,q$ of the Hamiltonian system $\mathcal H_{\text{III}'}$, see \cite{Okamoto}, in terms of 	$S_n=-p^1_n=\vartheta_{\eta}\tau_n$
		\begin{align*}
			q&=\frac{\eta\frac{\d^2S_n}{\d\eta^2}-2(n+b)\frac{\d S_n}{\d\eta}+2n}{4 \frac{\d S_n}{\d\eta}\big(1-\frac{\d S_n}{\d\eta}\big)},&
			p&=\frac{\d S_n}{\d\eta}
		\end{align*}
		with parameters $\theta_0=n+b$ and $\theta_\infty=n-b$.
		In terms of $\gamma_n=\eta\frac{\d S_n}{\d\eta}$ we have
		\begin{align*}
			q&=\frac{\eta^2\frac{\d\gamma_n}{\d\eta}-(2n+2b+1)\eta\gamma_n+2n\eta^2}{4 \gamma_n\big(\eta^2-\gamma_n\big)},&
			p&=\frac{\gamma_n}{\eta}.
		\end{align*}
		From \cite[\S 2.4]{clarkson} we conclude that $q$ satisfies $P_{\text{III}'}$  with $\alpha=-4(n-b)$, $\beta=4(n+b+1)$, $\gamma =4$ and $\delta=-4$, in \href{https://dlmf.nist.gov/32.2}{32.2.9} at DLMF, or Equation (2.4) in \cite{clarkson} with $A=-2\theta_\infty=-2(n-b)$ and $B=2(\theta_0+1)=2(n+b+1)$,
		\begin{align*}
			\frac{\d^2q}{\d\eta^2}=\frac{1}{q}\Big(\frac{\d q}{\d\eta}\Big)^2-\frac{1}{\eta}\frac{\d q}{\d\eta}-(n-b)\frac{q^2}{\eta^2}+\frac{n+b+1}{\eta}+\frac{q^3}{\eta}-\frac{1}{q}.
		\end{align*}
		Notice also \cite{clarkson,Okamoto} that $p^1_n$ is a solution to the $\text{P}_{\text{III}'}$  $\sigma$-equation. 
	\end{rem}
	
	\begin{rem}
		In terms of $p=\frac{\gamma_n}{\eta}=-\frac{\d p^1_n}{\d\eta}$, Equation \eqref{eq:ode_gamma_2} can be written as follows
		\begin{align*}
			\vartheta_{\eta}\Big(\vartheta_\eta^2 p-\frac{(\vartheta_\eta p)^2}{p}+2\eta p^2+\frac{n^2}{p}\Big)=2\eta p
		\end{align*}
		that expands to 
		\begin{align*}
			p^2	\vartheta_\eta^3 p-2p(\vartheta_\eta p)(\vartheta_\eta^2 p)+(\vartheta_\eta p)^3+(2\eta p^3-n^2)\vartheta_{\eta}p+2 p^3(p-\eta)=0.
		\end{align*}
		In \cite{clarkson} it is shown that $\frac{p-1}{p}$ satisfies  an instance of Painlevé V.
	\end{rem}
	\begin{rem}[The appearance of $\text{deg-P}_{\text{V}}$]
		After all these observations, one is  lead  to conjecture that Equation \eqref{eq:ode_gamma_2} should have the Painlevé property, and probably is solved in terms of the $\text{P}_{\text{III}}$ transcendents.
		This is indeed the case as was recently shown by Peter Clarkson in \cite{clarkson2} in where an integration of this equation leads to 
		\cite[Equation (83)]{clarkson2} which is equivalent to $\text{deg-P}_{\text{V}}$, i.e. \cite[Equation (38)]{clarkson2}.
	\end{rem}

\begin{pro}[Laguerre--Freud relations for the generalized Charlier case]\label{pro:Charlier Laguerre-Freud}
	For $n\in\N_0$, we find that the recursion coefficients satisfy the following  Laguerre--Freud  relations
	\begin{subequations}
		\begin{align}\label{eq:betgamma_charlier_1}
			\beta_{n+1}&=\frac{\eta (n+1)}{\gamma_{n+1}}-\beta_{n}+ n- b,\\
			\label{eq:equation2}	\gamma_{n+1}&=\eta-\gamma_n-\beta_n^2-b\beta_n+\frac{\gamma_{n-1}\gamma_n}{\eta}+\frac{\eta n^2}{\gamma_n},& \gamma_{-1}=\gamma_0=0.
		\end{align} 
	\end{subequations}
	For $n\in\N$, the following expression for the coefficient of the subleading term
	\begin{align}\label{eq:charlier_subleading}
		p^{1}_n=\frac{n(n+1)}{2}- n\beta_n
		- \frac{\gamma_n\gamma_{n+1}}{\eta}
	\end{align}
	holds, as a function of  near neighbors recursion coefficients holds.
\end{pro}
\begin{proof}
From \eqref{eq:Charlier_Psi_0} and \eqref{eq:Charlier_Psi_1}, using \eqref{eq:ladder_lambda},  we get two different expressions for the first superdiagonal involving only the recursion coefficients that we must equate, i.e.
\begin{align}\label{eq:charlier_intermedioHD}
	\eta HD&=H(\beta+{T_-}\beta+bI)\gamma -T_+\big(DH({T_-}\gamma )\gamma\big)
\end{align}
so that
\begin{align}
	\eta D\gamma^{-1}=\beta+{T_-}\beta+bI-(T_+D)\big(T_+{T_-}^2H\big)({T_-}H)^{-1}
	=\beta+{T_-}\beta+bI-T_+D,\label{eq:charlier_intermedioHD2}
\end{align}
where we used $\gamma=({T_-}H)H^{-1}$, that component wise is \eqref{eq:betgamma_charlier_1}.
Alternatively, notice that 
Equation \eqref{eq:a dos pasos}, after summing and dealing with  a telescopic series on the RHS,  gives   \eqref{eq:betgamma_charlier_1}. 

From the main diagonal and the second superdiagonal, again using \eqref{eq:ladder_lambda},
we  get the following   two expressions,
\begin{align}\label{eq:charlier_second}
	\eta H&= H(\gamma+T_+\gamma+\beta^2+b\beta)
	- T_+(D H(\beta+{T_-}\beta+bI)\gamma) +T_+^2\big(\pi^{[-2]}H({T_-}\gamma )\gamma\big),\\\label{eq:charlier_main}
	\eta H\pi^{[2]}&=H({T_-}\gamma )\gamma.
\end{align}
From \eqref{eq:charlier_main} and  \eqref{eq:pis} we obtain \eqref{eq:charlier_subleading}.
Again, using  $\gamma=({T_-}H)H^{-1}$ we get
\begin{align*}
	\eta &= 
		\gamma+T_+\gamma+\beta^2+b\beta
	- (T_+D) (T_+H)H^{-1}T_+(\beta+{T_-}\beta+bI)(T_+{T_-}H)T_+(H^{-1}) +T_+^2\big(\pi^{[-2]}{T_-}^2H\big)H^{-1}
\\
	&=\gamma+T_+\gamma+\beta^2+b\beta-(T_+D)T_+(\eta D\gamma^{-1}+T_+D)+T_+^2\big(\pi^{[-2]})\\
	&=\gamma+T_+\gamma+\beta^2+b\beta-(T_+D)(T_+^2D)-\eta(T_+D)^2T_+(\gamma^{-1})+T_+^2\pi^{[-2]}.
\end{align*}
Noticing that $2D^{[2]}=D{T_-}D$, we see that \eqref{eq:pis2} implies  $\pi^{[-2]}=D({T_-}D)-\pi^{[2]} $,
and \eqref{eq:charlier_main} gives $\pi^{[-2]}=D({T_-}D)-\eta^{-1}({T_-}^2H)H^{-1}$. Hence, $T_+^2\pi^{[-2]}=(T_+^2D)(T_+D)-\eta^{-1}HT_+(H^{-1})$ and we obtain
\begin{align*}
	\eta&=\gamma+T_+\gamma+\beta^2+b\beta-(T_+D)(T_+^2D)-\eta(T_+D)^2T_+(\gamma^{-1})+(T_+^2D)(T_+D)-\eta^{-1}HT_+(H^{-1})\\&=
	\gamma+T_+\gamma+\beta^2+b\beta-\eta(T_+D)^2T_+(\gamma^{-1})
	-\eta^{-1}HT_+^2(H^{-1}).
\end{align*}
We finally find
\begin{align*}
	\gamma+T_+\gamma+\beta^2+b\beta-\eta=\eta(T_+D)^2T_+(\gamma^{-1})
	+\eta^{-1}HT_+^2(H^{-1}),
\end{align*}
that component wise gives \eqref{eq:equation2}.
\end{proof}

\begin{rem}
	Smet \& Van Assche found,  see  \cite[Theorem 2.1]{smet_vanassche},  for the generalized Charlier weight that the corresponding  recursion coefficients  satisfy  (in the notation of this paper) \eqref{eq:betgamma_charlier_1} and, instead of \eqref{eq:equation2}, the following relation
\begin{align}\label{eq:smet_vanassche}
(\gamma_{n+1}-\eta)(\gamma_n-\eta)=\eta (\beta_n-n)(\beta_n-n+b).
\end{align}
Notice that \eqref{eq:smet_vanassche} is an equation of the form $\gamma_{n+1}=g(\gamma_n,\beta_n)$ which only involves a step backwards in $n$. In this sense, is better than  \eqref{eq:equation2} which is of the form $\gamma_{n+1}=f(\gamma_n,\gamma_{n-1},\beta_n)$. However, \eqref{eq:charlier_compatibility} gives $\gamma_{n-1}=F(\gamma_{n+1},\gamma_n,\beta_{n},\beta_{n-1})$, but \eqref{eq:betgamma_charlier_1}  gives $\beta_{n-1}=G(\beta_n,\gamma_n)$, so combining both we get $\gamma_{n-1}=F(\gamma_{n+1},\gamma_n,\beta_{n},G(\beta_n,\gamma_n))$ and we finally get a relation involving only $\gamma_{n+1},\gamma_{n}$ and $\beta_n$.
\end{rem}

\begin{pro}
If \eqref{eq:charlier_compatibility} holds then, 	equations \eqref{eq:betgamma_charlier_1} and \eqref{eq:equation2} are equivalent to \eqref{eq:betgamma_charlier_1}  and \eqref{eq:smet_vanassche}.
\end{pro}
\begin{proof}
	We rewrite \eqref{eq:charlier_compatibility} and \eqref{eq:betgamma_charlier_1} as indicated
	\begin{align}
\gamma_{n-1}&=	\gamma_{n+1}+	\frac{n	\eta (\beta_n-\beta_{n-1}-1)}{\gamma_n}, & n\geq 1.\\
\beta_{n-1}+1&=-\beta_{n}+\frac{\eta n}{\gamma_{n}}+ n- b.
	\end{align}
Hence,  for $n\geq 1$ and taking $\gamma_{-1}=\gamma_0=0$ we can write \eqref{eq:equation2}   as follows
\begin{align}\notag
	\gamma_{n+1}&=\eta-\gamma_n-\beta_n^2-b\beta_n+\frac{\gamma_n}{\eta}\Big(
	\gamma_{n+1}+	\frac{n	\eta (\beta_n-\beta_{n-1}-1)}{\gamma_n}
	\Big)+\frac{\eta n^2}{\gamma_n}\\\notag &=\eta+
	\frac{\gamma_n\gamma_{n+1}}{\eta}-\gamma_n+\frac{\eta n^2}{\gamma_n}-\beta_n^2-b\beta_n+n(\beta_n-\beta_{n-1}-1)\\\notag&=\eta+
		\frac{\gamma_n\gamma_{n+1}}{\eta}-\gamma_n+\frac{\eta n^2}{\gamma_n}-\beta_n^2-b\beta_n+n\big(2\beta_n-\frac{\eta n}{\gamma_{n}}- n+ b\big)\\\label{eq:gamman+1n}
		&=	\eta+\frac{\gamma_n\gamma_{n+1}}{\eta}-\gamma_n-\beta_n^2-b\beta_n+2n\beta_n- n^2+ bn.
\end{align}
that is \eqref{eq:smet_vanassche}.
 The inverse statement is easily proven to hold. If we assume \eqref{eq:smet_vanassche}, \eqref{eq:betgamma_charlier_1} and \eqref{eq:charlier_compatibility} to be true, we can go backwards in the chain of equalities leading to \eqref{eq:gamman+1n} and get the stated result. 
\end{proof}
\begin{coro}
	The $\beta$'s are subject to the nonlinear recursion
	\begin{multline}\label{eq:nonrecursion_beta_charlier}	
		\frac{\eta (n+1)}{\beta_{n+1}+\beta_{n}- n+ b}=\eta+\Big(\frac{ n-1}{\beta_{n-1}+\beta_{n-2}- n+2+ b}-1\Big)\frac{\eta n}{\beta_{n}+\beta_{n-1}- n+1+ b}-\beta_n^2-b\beta_n\\+n (\beta_{n}+\beta_{n-1}- n+1+ b).
	\end{multline} 
\end{coro}
\begin{proof}
	Notice that we can write  \eqref{eq:betgamma_charlier_1} as follows
		\begin{align}\label{eq:gamma->beta_charlier}
\gamma_{n+1}	&=\frac{\eta (n+1)}{\beta_{n+1}+\beta_{n}- n+ b},
	\end{align}
that we can introduce in \eqref{eq:equation2} to get \eqref{eq:nonrecursion_beta_charlier}.
\end{proof}

We now seek for a second order ODE for $\gamma_n$.
From the above  results it easily follows  that  (this was previously found in   \cite{filipuk_vanassche0}), 
\begin{lemma}
The recursion coefficients $\beta_n$ and $\gamma_n$ satisfy the following system of first order nonlinear  ODEs
\begin{subequations}\label{eq:Charlier_system_gamma_beta}
	\begin{align}\label{eq:Charlier_system_gamma_beta_1}
	\vartheta_{\eta}\beta_n&=\eta\frac{\eta+ (b-n)n+(2n-b-\beta_n)\beta_n-\gamma_n}{\eta-\gamma_n}-\gamma_n,\\
	\label{eq:Charlier_system_gamma_beta_2}
		\vartheta_\eta \gamma_n&=(b-n+1+2\beta_n)\gamma_n-n\eta.
\end{align}
\end{subequations}
\end{lemma}

\begin{proof}
	Equation \eqref{eq:gamman+1n}, after some cleaning, reads
	\begin{align}\label{eq:smet_vanassche_2}
		\gamma_{n+1}&=	\eta\frac{\eta+ (b-n)n+(2n-b-\beta_n)\beta_n-\gamma_n}{\eta-\gamma_n},
	\end{align}
which is  \cite[Equation (16)]{filipuk_vanassche0}.  
	In the one hand, from \eqref{eq:smet_vanassche_2} and the Toda equation \eqref{eq:Toda_system_beta} we find
	\eqref{eq:Charlier_system_gamma_beta_1}.
	That is Equation (18) in \cite{filipuk_vanassche0}.  
On the other hand, from \eqref{eq:betgamma_charlier_1} and the Toda equation \eqref{eq:Toda_system_gamma} we get
\eqref{eq:Charlier_system_gamma_beta_2}.
\end{proof}

\begin{rem}
	Differential system \eqref{eq:Charlier_system_gamma_beta} was used in \cite{filipuk_vanassche0} to get a second order nonlinear ODE for $\beta_n$ and then showed  \cite[Theorem 2.1]{filipuk_vanassche0} that an auxiliary function $y$, see    \cite[Equation (20)]{filipuk_vanassche0}, satisfies an instance of the $\text{P}_\text{V}$ related to the $\text{P}_\text{III}$. Notice that $y$ is a solution to a Ricatti equation in where $\beta_n$ appears in the coefficients. 
\end{rem}

\begin{theorem}[Second order ODE for generalized Charlier]\label{teo:second order}
The recursion coefficient $\gamma_n$ satisfies the second order nonlinear ODE
\begin{multline}\label{eq:edo_Charlier_2}
\Big(1-\frac{\gamma_n}{\eta}\Big)\Big(\vartheta_{\eta}\Big(
	\frac{\vartheta_\eta \gamma_n}{\gamma_n}+\frac{n\eta}{\gamma_n}\Big)+2\gamma_n\Big)+2(\gamma_n-\eta+(n-b)n)\\
=
	-\frac{1}{2}\Big(	\frac{\vartheta_\eta \gamma_n}{\gamma_n}+\frac{n\eta}{\gamma_n}\Big)^2+	(n+1)\Big(	
\frac{\vartheta_\eta \gamma_n}{\gamma_n}+\frac{n\eta}{\gamma_n}\Big)
	+	(-b+n-1)(-b+3n+1).
\end{multline}
\end{theorem}
\begin{proof}Observe  that Equation \eqref{eq:Charlier_system_gamma_beta_2} leads to
\begin{align}\label{eq:beta_n-gamma_n}
\beta_n=\frac{1}{2}\Big(	\vartheta_\eta \log \gamma_n+\frac{n\eta}{\gamma_n}-b+n-1\Big),
\end{align}
so that
\begin{align}\label{eq:diff-beta_n-gamma_n}
	\vartheta_{\eta}\beta_n=\frac{1}{2}\vartheta_{\eta}\Big(	\vartheta_\eta \log \gamma_n+\frac{n\eta}{\gamma_n}\Big).
\end{align}
From \eqref{eq:Charlier_system_gamma_beta_1} and \eqref{eq:diff-beta_n-gamma_n} we get
\begin{align*}
\Big(1-\frac{\gamma_n}{\eta}\Big)\Big(\frac{1}{2}\vartheta_{\eta}\Big(	\vartheta_\eta \log \gamma_n+\frac{n\eta}{\gamma_n}\Big)+\gamma_n\Big)+\gamma_n-\eta+(n-b)n=	(2n-b)\beta_n-\beta_n^2,
\end{align*}
and replacing $\beta_n$ by the expression provided in \eqref{eq:beta_n-gamma_n} we get \eqref{eq:edo_Charlier_2}.  To check it let us elaborate on  the RHS of the equation
\begin{align*}
	(2n-b)\beta_n-\beta_n^2&= 		\frac{2n-b}{2}\Big(	\vartheta_\eta \log \gamma_n+\frac{n\eta}{\gamma_n}-b+n-1\Big)
	-\frac{1}{4}\Big(	\vartheta_\eta \log \gamma_n+\frac{n\eta}{\gamma_n}-b+n-1\Big)^2\\&=
	 \begin{multlined}[t][0.75\textwidth]
	 		\frac{2n-b}{2}\Big(	\vartheta_\eta \log \gamma_n+\frac{n\eta}{\gamma_n}-b+n-1\Big)
	 	-\frac{1}{4}\Big(	\vartheta_\eta \log \gamma_n+\frac{n\eta}{\gamma_n}\Big)^2\\-\frac{(b-n+1)^2}{4}
	 	+\frac{b-n+1}{2}\Big(	\vartheta_\eta \log \gamma_n+\frac{n\eta}{\gamma_n}\Big)
	 \end{multlined}\\
 &=	 \begin{multlined}[t][0.75\textwidth]
 	-\frac{1}{4}\Big(	\vartheta_\eta \log \gamma_n+\frac{n\eta}{\gamma_n}\Big)^2+	\frac{n+1}{2}\Big(	\vartheta_\eta \log \gamma_n+\frac{n\eta}{\gamma_n}\Big)
+	\frac{(-b+n-1)(-b+3n+1)}{2},
 \end{multlined}
\end{align*}
and the statement is proven.
\end{proof}
Some additional properties of this generalized Charlier case follow.
\begin{pro}
	For the generalized Charlier case,  $\sigma=\eta$ and $\theta=z(z+b)$, the following holds:
	\begin{enumerate}
	\item The dressed Pascal matrix is
	\begin{align}\label{eq:Pascal_Charlier}
	\Pi=I+\Lambda^\top D+\eta^{-1}\big(\Lambda^\top)^2({T_-}\gamma)\gamma=
	\begin{bNiceMatrix}[columns-width = auto	]
		1 & 0 &0&\Cdots[shorten-end=5pt]\\
		1 & 1& 0&\Ddots\\
		\frac{\gamma_2\gamma_1}{\eta}&2&1&\Ddots\\
		0& \frac{\gamma_3\gamma_2}{\eta} &3 &\Ddots\\[5pt]
		\Vdots[shorten-start=5pt,shorten-end=5pt] &\Ddots[shorten-end=20pt]&\Ddots[shorten-end=5pt]&\Ddots[shorten-end=30pt]
	\end{bNiceMatrix}.
\end{align}
Moreover, the Jacobi and dressed Pascal matrices are linked by 
\begin{align}\label{eq:Pascal_Jacobi_Charlier}
	J^2+bJ= \eta\Pi H\Pi^\top H^{-1}.
\end{align}
\item The  corresponding orthogonal polynomials satisfy
	\begin{align*}
		\eta^{-1}&=\frac{P_{n+1}(0)P_n(-b)-P_{n+1}(-b)P_n(0)}{P_{n+2}(0)P_{n+1}(-b)-P_{n+2}(b)P_{n+1}(0)},&
		\frac{	n+1}{\gamma_{n+1}}&=	\frac{P_{n+2}(0)P_n(-b)-P_{n+2}(-b)P_n(0)}{P_{n+2}(0)P_{n+1}(-b)-P_{n+2}(-b)P_{n+1}(0)}.
	\end{align*}

\end{enumerate}
\end{pro}
\begin{proof}
	\begin{enumerate}
	\setcounter{enumi}{1}
	\item We  apply the ideas that leads to Christoffel formula for a Christoffel perturbation.
	From \eqref{eq:P_shift} we get $\theta(z) P(z-1)=\Psi H^{-1} P(z)=\eta H\Pi^\top H^{-1} P(z)$, where the last equation holds due  the generalized Charlier restrictions to the weight. 
Therefore,
\begin{align*}
\theta(z)H^{-1}P(z-1)=\Pi^\top 
H^{-1} P(z).
\end{align*}
As $\theta(0)=\theta(-b)=0$ we have
\begin{align*}\hspace*{-1cm}
\Pi^\top H^{-1} P(0)&=0,&
\Pi^\top H^{-1} P(-b)&=0.
\end{align*}
Both equations can be simplified to
\begin{align*}
[\begin{bNiceMatrix}\eta^{-1}\gamma_{n+2}\gamma_{n+1}&n+1\end{bNiceMatrix}=-\begin{bNiceMatrix}H_n^{-1}P_n(0)&H_n^{-1}P_n(-b)\end{bNiceMatrix}\begin{bNiceMatrix}
H_{n+2}^{-1}P_{n+2}(0) &H_{n+2}^{-1}P_{n+2}(-b) \\[2pt]H_{n+1}^{-1}P_{n+1}(0) &H_{n+1}^{-1}P_{n+1}(-b) 
\end{bNiceMatrix}^{-1},
\end{align*}
from where the result follows.
\end{enumerate}
\end{proof}

\section{Generalized Meixner  weights}\label{S:Meixner}
Another interesting case appears when one takes
 $\sigma=z+a$, that we call extended Meixner type. The generalized Meixner,   which is the main example of these extended Meixner weights,  corresponds to the choice  $\sigma(z)=\eta(z+a)$ and $\theta(z)=z(z+b)$. A weight that satisfies the  Pearson equation 
\begin{align*}
(k+1)(k+1+b)w(k+1)=\eta (k+a)w(k)
\end{align*}
is proportional to the generalized Meixner weight 
\begin{align*}
w(z)=\frac{\Gamma(b+1)\Gamma(z+a)\eta^z}{\Gamma(a)\Gamma(z+b+1)\Gamma(z+1)}=\frac{(a)_z}{(b+1)_z}\frac{\eta^z}{z!}.
\end{align*}
The finiteness of the moments requires $a(b+1)>0$ and $\eta>0$ \cite{diego_paco}.
In this case the moments are expressed in terms of 
$\rho_0={}_1F_1(a;b+1;\eta)=M(a,b+1,\eta)$,
also known as the confluent hypergeometric function or Kummer function.

\begin{rem}
	From previous comments and  \eqref{eq:tau_recursion}  we get that in terms of the following  $\delta$-Wronskian  of the Kummer function
	\begin{align*}
		\tau_n&:=\mathscr W_{n}\left(M(a,b+1,\eta)\right),
	\end{align*}
	we get explicit expressions for the recursion coefficients
	\begin{align*}
		\beta_n&=\vartheta_{\eta}\log \frac{\tau_{n+1}}{\tau_n},&  \gamma_{n+1}&=\frac{\tau_{n+1}\tau_{n-1}}{\tau_n^2},& n\in\N_0.
	\end{align*}
\end{rem}

\begin{rem}
	Meixner introduced the non generalized version with $w=(a)_z\frac{\eta^z}{z!}$ in \cite{meixner}.
\end{rem}

\begin{rem}
	For an extended Meixner type weight,  as $M=\deg\sigma=1$ we have that the Laguerre--Freud matrix $\Psi=\Pi^{-1}H\theta(J^\top)=\eta(J+a)H\Pi^\top$  is a banded matrix with only one subdiagonal and $N$ superdiagonals. Now we do not have, as in the extended Charlier case that
$\Psi=H\Pi^\top$, which implied that the dressed Pascal matrix $\Pi$ had only three nonzero subdiagonals. In the extended Meixner case the dressed Pascal  matrix will possibly have an infinite number of nonzero subdiagonals.
\end{rem}
\begin{theorem}[The generalized Meixner Laguerre--Freud  structure matrix]\label{teo:generalized Meixner}
	For a generalized Meixner weight; i.e. $\sigma=\eta(z+a)$ and $\theta=z(z+b)$,  the structure matrix is
\begin{align}\label{eq:Laguerre-Freud-Meixner-structure}
	\hspace*{-.5cm}\Psi&=\left[\begin{NiceMatrix}[columns-width = .4cm	]
		\eta(\beta_0+a)H_0 & (\beta_0+\beta_1+b)H_1& H_2& 0 &\Cdots&\\
		\eta H_1&\eta (\beta_1+a+1)H_1& (\beta_1+\beta_2+b-1)H_2&H_3&\Ddots&\\[5pt]
		0&\eta H_2 &\eta (\beta_2+a+2)H_2& (\beta_2+\beta_3+b-2)H_3&H_4&
		\\[4pt]
		\Vdots[shorten-end=7pt]&\Ddots&\eta H_3 &\eta(\beta_3+a+3)H_3& (\beta_3+\beta_4+b-3)H_4&\Ddots
		\\
		& &\Ddots&\Ddots[shorten-end=70pt]&\Ddots[shorten-end=70pt]&\Ddots
	\end{NiceMatrix}\right].
\end{align}
\end{theorem}
\begin{proof}
The structure matrix has a diagonal structure, see \eqref{eq:diagonals_Psi},
$\Psi=\Lambda^\top\psi^{(-1)}+\psi^{(0)}+\psi^{(1)}\Lambda+
\psi^{(2)}\Lambda^2$.
As $\Psi=\eta(J+aI)H\Pi^\top$ we  find
\begin{align}\label{eq:structure_Meixner_1}
\begin{aligned}
\Psi&=\eta(\Lambda^\top\gamma+\beta+aI+\Lambda)H(I+D\Lambda+\pi^{[2]}\Lambda^2+\pi^{[3]}\Lambda^3+\cdots)\\&=
\begin{multlined}[t][.9\textwidth]
\underset{\text{first subdiagonal}}{\underbrace{\eta\Lambda^\top\gamma H}}+\underset{\text{main diagonal}}{\underbrace{\eta(\Lambda^\top\gamma H D\Lambda+(\beta+a I) H)}}+\underset{\text{first superdiagonal}}{\underbrace{\eta(\Lambda^\top\gamma H \pi^{[2]}\Lambda^2+(\beta+a I) HD\Lambda+\Lambda H)}}\\+
\underset{\text{second superdiagonal}}{\underbrace{\eta(\Lambda^\top\gamma H \pi^{[3]}\Lambda^3+(\beta+a I) H\pi^{[2]}\Lambda^2+\Lambda HD\Lambda)}}
+
\cdots
\end{multlined}
\end{aligned}
\end{align}
we get
$\psi^{(0)}=\eta T_+(\gamma H D)+\eta(\beta+a I) H$.
Now, as
\begin{align*}
J^2+b J=(\Lambda^\top)^2 ({T_-}\gamma )\gamma+ \Lambda^\top (\beta+{T_-}\beta+bI)\gamma+\gamma+T_+\gamma+\beta^2+b\beta+(\beta+{T_-}\beta +bI)\Lambda+\Lambda^2,
\end{align*}
from the alternative expression $\Psi=\Pi^{-1}H\big((J^\top)^2+bJ^\top\big)$,  we find
{\small\begin{align}\label{eq:structure_Meixner2}
\hspace*{-1.5cm}\begin{aligned}
\Psi&=
\begin{multlined}[t][\textwidth]
(I-\Lambda^\top D+(\Lambda^\top)^2 \pi^{[-2]}-(\Lambda^\top)^3 \pi^{[-3]}+\cdots)H\\\times \big((\Lambda^\top)^2+
\Lambda^\top (\beta+{T_-}\beta+bI)+\gamma+T_+\gamma+\beta^2+b\beta+ (\beta+{T_-}\beta+bI)\gamma \Lambda+({T_-}\gamma )\gamma\Lambda^2\big)
\end{multlined}\\&=
\begin{multlined}[t][\textwidth]
\underset{\text{second superdiagonal}}{\underbrace{H({T_-}\gamma )\gamma\Lambda^2}}\underset{\text{first superdiagonal}}{\underbrace{- \Lambda^\top D H({T_-}\gamma )\gamma\Lambda^2+H(\beta+{T_-}\beta+bI)\gamma \Lambda}}\\+\underset{\text{main diagonal}}{\underbrace{
H(\gamma+T_+\gamma+\beta^2+b\beta)-\Lambda^\top DH(\beta+{T_-}\beta+bI)\gamma \Lambda+(\Lambda^\top)^2 \pi^{[-2]}H({T_-}\gamma )\gamma\Lambda^2}}\\\hspace*{-1cm}+
\underset{\text{first subdiagonal}}{\underbrace{H\Lambda^\top (\beta+{T_-}\beta+bI)-\Lambda^\top DH(\gamma+T_+\gamma+\beta^2+b\beta)+(\Lambda^\top)^2 \pi^{[-2]}H(\beta+{T_-}\beta+bI)\gamma \Lambda-(\Lambda^\top)^3 \pi^{[-3]}H({T_-}\gamma )\gamma\Lambda^2.}}
\end{multlined}
\end{aligned}
\end{align}}

Thus,
$\psi^{(1)}=- T_+(D H({T_-}\gamma )\gamma)+H(\beta+{T_-}\beta+bI)\gamma $.
Finally, the Laguerre--Freud matrix has the form
\begin{align*}
\Psi&=\eta\Lambda^\top {T_-} H +\eta(\beta+a I+T_+D)H +\big(
\beta+{T_-}\beta+bI- T_+D \big){T_-}H\Lambda+{T_-}^2H\Lambda^2.
\end{align*}
When expressed component wise we get the given form in \eqref{eq:Laguerre-Freud-Meixner-structure}. 
\end{proof}

\begin{pro}[Compatibility]
	For the  generalized Meixner case the recursion coefficients fulfill
\begin{subequations}
		\begin{align}\label{eq:laguerre-freud-meixner-2}
	(\beta_n+\beta_{n+1}+b-n-\eta)\gamma_{n+1}-(\beta_{n-1}+\beta_n+b-n+1-\eta)\gamma_n&=\eta(\beta_n+a+n),\\\label{eq:laguerre-freud-meixner-3}
		\gamma_{n+2}-\gamma_n-2\eta	+(\beta_n+\beta_{n+1}+b-n-\eta)(\beta_{n+1}-\beta_n-1)&=0.
	\end{align}
\end{subequations}
\end{pro}

\begin{proof}
	The matrix
\begin{align*}
	\hspace*{-1cm}\Psi H^{-1}&=\left[\begin{NiceMatrix}[columns-width = .3cm	]
		\eta(\beta_0+a)& \beta_0+\beta_1+b& 1& 0 &\Cdots&\\
		\eta \gamma_1&\eta (\beta_1+a+1)& \beta_1+\beta_2+b-1&1&\Ddots&\\[5pt]
		0&\eta \gamma_2 &\eta (\beta_2+a+2)& \beta_2+\beta_3+b-2&1&
		\\[4pt]
		\Vdots[shorten-start=5pt,shorten-end=5pt]&\Ddots&\eta \gamma_3 &\eta(\beta_3+a+3)& \beta_3+\beta_4+b-3&\Ddots
		\\
		& &&\Ddots[shorten-end=50pt]&\Ddots[shorten-end=40pt]&\Ddots
	\end{NiceMatrix}\right]
\end{align*}
satisfies the compatibility condition \eqref{eq:compatibility_Jacobi_structure_a}. 
As we have
{\begin{align*}
		\hspace*{-1cm}\big[\Psi H^{-1}, J\big]&=\left[\small\begin{NiceMatrix}[small,columns-width = 5pt	]
			\cellcolor{Gray!10}	(\beta_0+\beta_1+b-\eta)\gamma_1& {\cellcolor{Gray!25}\scriptsize\begin{matrix}
					\gamma_2-\eta\\+	(\beta_0+\beta_1+b-\eta)(\beta_1-\beta_0)
			\end{matrix}}& 1& 0 &\Cdots\\
			\eta \gamma_1&{\cellcolor{Gray!10}\scriptsize\begin{matrix}
					(\beta_1+\beta_2+b-1-\eta)\gamma_2\\-(\beta_0+\beta_1+b-\eta)\gamma_1
			\end{matrix}}& {\cellcolor{Gray!25}\scriptsize\begin{matrix}
					\gamma_3-\gamma_1-\eta\\+(\beta_1+\beta_2+b-1-\eta)(\beta_2-\beta_1)
			\end{matrix}}&1&\Ddots\\[5pt]
			0&\eta \gamma_2 &{\cellcolor{Gray!10}\scriptsize\begin{matrix}
					(\beta_2+\beta_3+b-2-\eta)\gamma_3\\-(\beta_1+\beta_2+b-1-\eta)\gamma_2
			\end{matrix}}&  \cellcolor{Gray!25}{\scriptsize\begin{matrix}
					\gamma_4-\gamma_2-\eta\\	+(\beta_2+\beta_3+b-2-\eta)(\beta_3-\beta_2)
			\end{matrix}}&
			\\
			\Vdots[shorten-end=5pt]	&\Ddots[shorten-end=45pt]	 &\Ddots[shorten-end=120pt]	&\Ddots[shorten-end=120pt]	&\Ddots
		\end{NiceMatrix}\right]
\end{align*}}
we get for the first row
	\begin{align*}
\gamma_1&=	\frac{\eta(\beta_0+a)}{\beta_0+\beta_1+b-\eta}, &
\gamma_2&=
 \beta_0+\beta_1+b-	(\beta_0+\beta_1+b-\eta)(\beta_1-\beta_0)+\eta,
	\end{align*}
while for the next rows we obtain
\begin{align*}
	(\beta_n+\beta_{n+1}+b-n-\eta)\gamma_{n+1}-(\beta_{n-1}+\beta_n+b-n+1-\eta)\gamma_n&=\eta(\beta_n+a+n),\\
	\gamma_{n+2}-\gamma_n-\eta	+(\beta_n+\beta_{n+1}+b-n-\eta)(\beta_{n+1}-\beta_n)&=\beta_n+\beta_{n+1}+b-n.
\end{align*}
Which are \eqref{eq:laguerre-freud-meixner-2} and \eqref{eq:laguerre-freud-meixner-3} (in disguise the last one, some clearing is required).

An alternative manner to get \eqref{eq:laguerre-freud-meixner-2} follows. The matrix $\Psi H^{-1}$  fulfills \eqref{eq:eta_compatibility_Pearson_1b}, where $\Phi=-J_-$, and we have
$	\vartheta_{\eta}(\Psi H^{-1})=[\Psi H^{-1},J_-]$.
The explicit expression of the above commutator is:
\begin{align*}
	\hspace*{-1cm}\big[\Psi H^{-1}, J_-\big]&=\left[\begin{NiceMatrix}[columns-width = 1cm,,]
		\cellcolor{Gray!15}	(\beta_0+\beta_1+b)\gamma_1& 	\gamma_2& 0& 0&\Cdots&\\
		\eta \gamma_1(\beta_1-\beta_0+1)&\cellcolor{Gray!15}\scriptsize\begin{matrix}
			(\beta_1+\beta_2+b-1)\gamma_2\\-(\beta_0+\beta_1+b)\gamma_1	\end{matrix}&\gamma_3-\gamma_1&0&\Ddots&\\[5pt]
		0& 	\eta \gamma_2 (\beta_2-\beta_1+1)&{\cellcolor{Gray!15}\scriptsize\begin{matrix}
				(\beta_2+\beta_3+b-2)\gamma_3\\-(\beta_1+\beta_2+b-1)\gamma_2
		\end{matrix}}&  
		\gamma_4-\gamma_2&
		\\
		\Vdots[shorten-start=5pt, shorten-end=10pt]	& &\Ddots[shorten-end=50pt]&\Ddots[shorten-end=20pt]&\Ddots
	\end{NiceMatrix}\right],
\end{align*}
and, consequently, we find the following  equations
\begin{align}
\label{eq:1}\vartheta_\eta (\eta\gamma_n)&=\eta\gamma_n(\beta_n-\beta_{n-1}+1),\\
\label{eq:2}\vartheta_{\eta}(\beta_n+\beta_{n+1}+b-n)&=\gamma_{n+2}-\gamma_n,\\
\label{eq:3}
\vartheta_\eta\big(\eta(\beta_n+a+n)\big)
&=\gamma_{n+1}(\beta_n+\beta_{n+1}+b-n)
	-\gamma_n(\beta_{n-1}+\beta_n+b-(n-1)).
\end{align}
Notice that \eqref{eq:1} and \eqref{eq:2} follow from the Toda equations for the recursion coefficients \eqref{eq:Toda_system}. The Toda equation  \eqref{eq:Toda_system_beta} allows to write \eqref{eq:3} as follows
\begin{align*}
\eta(\beta_n+a+n)+\eta\gamma_{n+1}-\eta\gamma_n
&=\gamma_{n+1}(\beta_n+\beta_{n+1}+b-n)
-\gamma_n(\beta_{n-1}+\beta_n+b-(n-1)),
\end{align*}
and we get \eqref{eq:laguerre-freud-meixner-2}.
\end{proof}

\begin{pro}[Laguerre--Freud relations]
	For $n\in\N$, we find the following Laguerre--Freud equations
\begin{align}\label{eq:Laguerre-Freud-Meixner1}
	\gamma_{n+1}&=	-\gamma_{n}-(\beta_{n}+b-n)\beta_{n}+n(b-a-n+1) +\eta(\beta_{n}+a+n)+\eta^{-1}(\beta_{n-1}+\beta_{n}+b-n+1-\eta) \gamma_{n},\\\label{eq:meixner_laguerre_freud_larga}
0&=\begin{aligned}[t]
	&(n+3)(n+2)(n+1)(\beta_n-\beta_{n+2}+1)-\big(\eta^{-1}\gamma_{n+3} -n-3\big)\gamma_{n+2}\\&+(n+3)(n+2)
\Big(\eta^{-1}(\beta_{n+1}+\beta_{n+2}+b-n-1-\eta) \gamma_{n+2}-(n+2)(\beta_{n+1}+a )  \Big)
\\&-(n+2)(n+1)\Big(\eta^{-1}(\beta_{n+3}+\beta_{n+4}+b-n-3-\eta) \gamma_{n+4}-(n+4)(\beta_{n+3}+a )\Big)\\
&	+(-\beta_{n+2}+a-b ) \big(
\eta^{-1}(\beta_{n+2}+\beta_{n+3}+b-n-2-\eta)\gamma_{n+3}	-	(n+3)	(\beta_{n+2}+a ) \big)  \\
&+(\beta_{n+2}+\beta_{n+3}+b-n-3-\eta)\gamma_{n+3}-(n+3)(\gamma_{n+2}+\beta_{n+2}^2+b\beta_{n+2})\\&+(n+3)(n+2)(\beta_{n+1}+\beta_{n+2}+b).
\end{aligned}
\end{align}
We also find
	\begin{align}\label{eq:Laguerre-Freud-Meixner3}
	p^{1}_{n+1}&=-\eta^{-1}(\beta_{n+1}+\beta_{n+2}+b-n-1-\eta) \gamma_{n+2}+\beta_{n+1}+a (n+2) +\frac{(n+2)(n+1)}{2}.
	\end{align}
\end{pro}
\begin{proof}
	Comparing \eqref{eq:structure_Meixner_1} with \eqref{eq:structure_Meixner2} and  recalling \eqref{eq:ladder_lambda} we obtain
	\begin{gather}\label{eq:hyper_meixner_1}\eta {T_-}H= 
\begin{multlined}[t][.7\textwidth]
	(\beta+{T_-}\beta+bI)T_-H- DH(\gamma+T_+\gamma+\beta^2+b\beta)\\+T_+\big( \pi^{[-2]}(\beta+{T_-}\beta+bI)H\gamma\big) +T_+^2\big(\pi^{[-3]}H({T_-}\gamma )\gamma\big),
\end{multlined}\\\label{eq:hyper_meixner_2}
	\eta(T_+(\gamma H D)+(\beta+a I) H)=
	\begin{multlined}[t][.5\textwidth]H(\gamma+T_+\gamma+\beta^2+b\beta)\\-T_+(DH(\beta+{T_-}\beta+bI)\gamma) +T_+^2 (\pi^{[-2]}H({T_-}\gamma )\gamma),
	\end{multlined}\\\label{eq:hyper_meixner_3}
		\eta\big(T_+(\gamma H \pi^{[2]})+(\beta+a I)H D+{T_-}H\big)=-  T_+\big(D H({T_-}\gamma )\gamma\big)+H(\beta+{T_-}\beta+bI)\gamma ,\\\label{eq:hyper_meixner_4}
		\eta\big(T_+(\gamma H \pi^{[3]})+(\beta+a I) H\pi^{[2]}+{T_-} (HD)\big)=H({T_-}\gamma )\gamma.
	\end{gather}
	Let us study relations \eqref{eq:hyper_meixner_2} and \eqref{eq:hyper_meixner_3}. 
		In the one hand, if we choose to solve  \eqref{eq:hyper_meixner_3} for $\pi^{[2]}$, we get (recall that ${T_-}T_+=\text{id}$)
	\begin{align}\notag
		\pi^{[2]}&=-H^{-1}\gamma^{-1} {T_-}((\beta+a I) HD+{T_-}H)- \eta^{-1}D {T_-}\gamma + \eta^{-1}H^{-1}\gamma^{-1}{T_-}(H\gamma(\beta+{T_-}\beta+bI)) \\&=
		\eta^{-1}({T_-}\beta+T_-^2\beta+bI-D-\eta) {T_-}\gamma-({T_-}D)({T_-}\beta+aI),\label{eq:pi2}
	\end{align}
	so that
	\begin{align*}
		\pi^{[2]}_n&=\eta^{-1}(\beta_{n+1}+\beta_{n+2}+b-n-1-\eta) \gamma_{n+2}-(n+2)(\beta_{n+1}+a ) .
	\end{align*}
Moreover, from 
$\beta=T_+S^{[1]}-S^{[1]}$ and 
$\pi^{[ 2]}=D^{[2]}+
(T_-S^{[1]}-S^{[1]})D- S^{[1]}$ we get
\begin{align}\label{eq:S1_Meixner}
	S^{[1]}&=-\pi^{[2]}+D^{[2]}-DT_-\beta
\end{align}
that,  component wise,
is \eqref{eq:Laguerre-Freud-Meixner3}.
	Now, recalling \eqref{eq:pis2}, we write \eqref{eq:hyper_meixner_2} as follows
\begin{gather}\label{eq:hyper_meixner_2_bis}
	\eta (\beta+a I+T_+D)
	=\gamma+T_+\gamma+\beta^2+b\beta-T_+(D(\beta+{T_-}\beta+bI)) +T_+^2 (-\pi^{[2]}+2D^{[2]}),
	\end{gather}
that is, using \eqref{eq:theDs}, 
	\begin{align*}
		\eta(T_-^2\beta+a I+T_- D) &=T_-^2\gamma+{T_-}\gamma+T_-^2\beta^2+bT_-^2\beta-(T_-D)({T_-}\beta+T_-^2\beta+bI-D) -\pi^{[2]}\\&=
\begin{multlined}[t][0.75\textwidth]
		T_-^2\gamma+{T_-}\gamma+T_-^2\beta^2+bT_-^2\beta-(T_-D)({T_-}\beta+T_-^2\beta+bI-D)\\-
				\eta^{-1}({T_-}\beta+T_-^2\beta+bI-D-\eta) {T_-}\gamma+(T_-D)({T_-}\beta+a I)                 	
\end{multlined}\\&=\begin{multlined}[t][0.75\textwidth]
T_-^2\gamma+{T_-}\gamma+T_-^2\beta^2+bT_-^2\beta-(T_-D)(T_-^2\beta+(b-a)I-D)\\-
\eta^{-1}({T_-}\beta+T_-^2\beta+bI-D-\eta) {T_-}\gamma,
\end{multlined}	\end{align*}
	that component wise reads
	\begin{multline*}
	\eta(\beta_{n+2}+a +n+2) =\gamma_{n+3}+\gamma_{n+2}+\beta_{n+2}^2+b\beta_{n+2}-(n+2)(\beta_{n+2}+b-a-n-1) \\-\eta^{-1}(\beta_{n+1}+\beta_{n+2}+b-n-1-\eta) \gamma_{n+2},
\end{multline*}
	which is \eqref{eq:Laguerre-Freud-Meixner1}.
	
	Now, \eqref{eq:hyper_meixner_4} can be written as follows
	\begin{align}\label{eq:pi3}
\begin{aligned}
				\pi^{[3]}&=-({T_-}\beta+a I) {T_-}\pi^{[2]}+\big(\eta^{-1}T_-^2\gamma -T_-^2D\big)T_-\gamma\\&=
			\begin{multlined}[t][\textwidth]
			-	(T_-\beta+a I) \big(
	\eta^{-1}(T_-^2\beta+T_-^3\beta+bI-T_-D-\eta)T_-^2\gamma	-	(T_-^2D)	(T_-^2\beta+a I) \big)  +\big(\eta^{-1}T_-^2\gamma -T_-^2D\big)T_-\gamma,
			\end{multlined}
\end{aligned}
	\end{align}
where \eqref{eq:pi2} has been used. 
	From \eqref{eq:pis2}, $\pi^{[3]}+\pi^{[-3]}=2(({T_-}^2S^{[1]})D^{[2]}-({T_-}D^{[2]})S^{[1]})$, and   
\eqref{eq:S1_Meixner}  (noticing that $3D^{[3]}=2D^{[2]}(T_-^2D^{[2]}-T_-D^{[2]})$)  we get
	\begin{align*}
	\pi^{[-3]}=&2D^{[2]}(-T_-^2\pi^{[2]}+T_-^2D^{[2]}-(T_-^2D)T_-^3\beta)-
		2(T_-D^{[2]})(-\pi^{[2]}+D^{[2]}-DT_-\beta)-	\pi^{[3]}\\
		=&3D^{[3]}+2({T_-}D^{[2]})(\pi^{[2]}+DT_-\beta)-2D^{[2]}(T_-^2\pi^{[2]}+(T_-^2D)T_-^3\beta)-	\pi^{[3]}\\
		=&3D^{[3]}T_-(\beta-{T_-}^2\beta+I)+2({T_-}D^{[2]})\pi^{[2]}-2D^{[2]}T_-^2\pi^{[2]}-	\pi^{[3]}\\
		=&3D^{[3]}T_-(\beta-T_-^2\beta+I)-\big(\eta^{-1}T_-^2\gamma -T_-^2D\big)T_-\gamma\\&+2(T_-D^{[2]})
		\Big(\eta^{-1}(T_-\beta+T_-^2\beta+bI-D-\eta) T_-\gamma-(T_-D)(T_-\beta+a I)  \Big)
		\\&-2D^{[2]}\Big(\eta^{-1}(T_-^3\beta+T_-^4\beta+bI-T_-^2D-\eta) T_-^3\gamma-(T_-^3D)(T_-^3\beta+a I)\Big)\\
		&	+(T_-\beta+a I) \big(
		\eta^{-1}(T_-^2\beta+T_-^3\beta+bI-T_-D-\eta)T_-^2\gamma	-	(T_-^2D)	(T_-^2\beta+a I) \big).
	\end{align*}
	
	We now write \eqref{eq:hyper_meixner_1} 
	as follows
	\begin{align*}
		\pi^{[-3]}=& 
		\begin{multlined}[t][\textwidth]
		-(T_-^2\beta+T_-^3\beta+bI-T_-^2D-\eta)T_-^2\gamma+ (T_-^2D)(T_-\gamma+T_-^2\beta^2+bT_-^2\beta)-({T_-}\beta+T_-^2\beta+bI)
			T_-(2D^{[2]}-\pi^{[2]})
		\end{multlined}\\
=&	\begin{multlined}[t][\textwidth]-(T_-^2\beta+T_-^3\beta+bI-T_-^2D-\eta)T_-^2\gamma+(T_-^2D)(T_-\gamma+T_-^2\beta^2+bT_-^2\beta)\\-({T_-}\beta+T_-^2\beta+bI)		T_-\big(2D^{[2]}-	\eta^{-1}(T_-\beta+T_-^2\beta+bI-D-\eta) T_-\gamma+({T_-}D)({T_-}\beta+a I)  \big).  	\end{multlined}
 \end{align*}
Thus, we get the following equation
\begin{align*}
0=&3D^{[3]}(T_-\beta-T_-^3\beta+I)-\big(\eta^{-1}T_-^2\gamma -T_-^2D\big)T_-\gamma\\&+2(T_-D^{[2]})
	\Big(\eta^{-1}(T_-\beta+T_-^2\beta+bI-D-\eta) T_-\gamma-(T_-D)(T_-\beta+a I)  \Big)
	\\&-2D^{[2]}\Big(\eta^{-1}(T_-^3\beta+T_-^4\beta+bI-T_-^2D-\eta) T_-^3\gamma-(T_-^3D)(T_-^3\beta+a I)\Big)\\
	&	+(T_-\beta+a I) \big(
	\eta^{-1}(T_-^2\beta+T_-^3\beta+bI-T_-D-\eta)T_-^2\gamma	-	(T_-^2D)	(T_-^2\beta+a I) \big)  \\
	&+(T_-^2\beta+T_-^3\beta+bI-T_-^2D-\eta)T_-^2\gamma-(T_-^2D)(T_-\gamma+T_-^2\beta^2+bT_-^2\beta)\\&+({T_-}\beta+T_-^2\beta+bI)		T_-\big(2D^{[2]}-	\eta^{-1}(T_-\beta+T_-^2\beta+bI-D-\eta) T_-\gamma+({T_-}D)({T_-}\beta+a I)  \big).  
\end{align*}
Hence, component wise we get \eqref{eq:meixner_laguerre_freud_larga}.
\end{proof}
\begin{rem}
	Notice that the long and cumbersome expression in \eqref{eq:meixner_laguerre_freud_larga} is a ligature for
the recursion coefficients
 \begin{align*}
	(\beta_{n+4},\beta_{n+2},\beta_{n+1},\beta_n,\gamma_{n+4},\gamma_{n+3},\gamma_{n+2}). 
\end{align*}
We just wanted to show that the method gives these equations. Indeed,  \eqref{eq:Laguerre-Freud-Meixner1} is a  better  Laguerre--Freud equation for the $\beta$'s. However, notice that   \eqref{eq:Laguerre-Freud-Meixner1} implies an expression of $\beta_{n+1}$ in terms  $(\beta_{n-1},\beta_n,\gamma_{n+1},\gamma_n)$ while   \eqref{eq:laguerre-freud-meixner-2} 
involves and expression for $\beta_{n}$ in terms  of $(\beta_n,\beta_{n-1},\gamma_n,\gamma_{n+1})$. In both cases we need to go two steps down,  and we say that we have \emph{length two relations}. However, we can mix both Laguerre--Freud equations to find an \emph{improved} expression for $\beta_{n+1}$ involving only one step down, i.e. $(\beta_n,\gamma_n,\gamma_{n+1})$, a \emph{length one relation}.
A similar reasoning holds for \eqref{eq:laguerre-freud-meixner-3}.
\end{rem}
\begin{pro}
	The generalized Meixner recursion coefficients are subject to the following length one Laguerre--Freud equation
	\begin{multline}\label{eq:Meixner_Laguerre_Freud_step1}
		\beta_{n+1}=
			\frac{1}{\gamma_{n+1}}\Big(\eta \big(\gamma_{n}+(\beta_{n}+b-n)\beta_{n}-n(b-a-n+1) \big)+\eta(\eta+1)(\beta_{n}+a+n)\Big)-\beta_n-b+n+1+2\eta.
	\end{multline}
\end{pro}

\begin{proof}
	We write  \eqref{eq:Laguerre-Freud-Meixner1} as follows
		\begin{multline*}
(\beta_{n-1}+\beta_{n}+b-n+1-\eta) \gamma_{n}=
\eta \big(\gamma_{n+1}+\gamma_{n}+(\beta_{n}+b-n)\beta_{n}-n(b-a-n+1) \big)+\eta^2(\beta_{n}+a+n)
\end{multline*}
that, when introduced  in \eqref{eq:laguerre-freud-meixner-2}, delivers
\begin{align*}
	(\beta_n+\beta_{n+1}+b-n-\eta)\gamma_{n+1}=&(\beta_{n-1}+\beta_n+b-n+1-\eta)\gamma_n+\eta(\beta_n+a+n)\\=&
\begin{multlined}[t][0.65\textwidth]\eta \big(\gamma_{n+1}+\gamma_{n}+(\beta_{n}+b-n)\beta_{n}-n(b-a-n+1) \big)
	+\eta(\eta+1)(\beta_{n}+a+n),
\end{multlined}
\end{align*}
and we obtain the Laguerre--Freud equation \eqref{eq:Meixner_Laguerre_Freud_step1}.

Let us shift by $n\to n+1$ the relation \eqref {eq:Laguerre-Freud-Meixner1} 
\begin{multline*}
	\gamma_{n+2}=	-\gamma_{n+1}-(\beta_{n+1}+b-n-1)\beta_{n+1}+n(b-a-n) \\+\eta(\beta_{n+1}+a+n+1)+\eta^{-1}(\beta_{n}+\beta_{n+1}+b-n-\eta) \gamma_{n+1},
\end{multline*}
and introduce the expression for $\gamma_{n+2}$ into \eqref{eq:laguerre-freud-meixner-3}. In doing so we obtain
\begin{multline*}
	\gamma_n+2\eta	-(\beta_n+\beta_{n+1}+b-n-\eta)(\beta_{n+1}-\beta_n-1)=-(\beta_{n+1}+b-n-1)\beta_{n+1}+n(b-a-n) \\+\eta(\beta_{n+1}+a+n+1)+\eta^{-1}(\beta_{n}+\beta_{n+1}+b-n-2\eta) \gamma_{n+1},
\end{multline*}
but  \eqref{eq:Meixner_Laguerre_Freud_step1} can be written
 	\begin{multline*}
 	\gamma_{n+1}(\beta_{n+1}+\beta_n+b-n-2\eta)=
 	\eta \big(\gamma_{n}+(\beta_{n}+b-n)\beta_{n}-n(b-a-n+1) \big)+\eta(\eta+1)(\beta_{n}+a+n)+\gamma_{n+1},
 \end{multline*}
so that
\begin{multline*}
	\gamma_n+2\eta	-(\beta_n+\beta_{n+1}+b-n-\eta)(\beta_{n+1}-\beta_n-1)=-(\beta_{n+1}+b-n-1)\beta_{n+1}+n(b-a-n) \\+\eta(\beta_{n+1}+a+n+1)+\big(\gamma_{n}+(\beta_{n}+b-n)\beta_{n}-n(b-a-n+1) \big)\\+(\eta+1)(\beta_{n}+a+n)+\eta^{-1} \gamma_{n+1}.
\end{multline*}
\end{proof}

Our \emph{best} Laguerre--Freud equations are \eqref{eq:Meixner_Laguerre_Freud_step1} and \eqref{eq:Laguerre-Freud-Meixner1}, having lengths one and two, respectively. For the reader convenience we reproduce them again as a Theorem

\begin{theorem}[Laguerre--Freud equations for generalized Meixner]\label{teo:Laguerre-Freud generalized Meixner}
The generalized Meixner recursion coefficients satisfy the following Laguerre--Freud relations
\begin{align*}
		\beta_{n+1}&=\begin{multlined}[t][0.75\textwidth]
	\frac{1}{\gamma_{n+1}}\Big(\eta \big(\gamma_{n}+(\beta_{n}+b-n)\beta_{n}-n(b-a-n+1) \big)+\eta(\eta+1)(\beta_{n}+a+n)\Big)-\beta_n-b+n+1+2\eta,
	\end{multlined}\\
	\gamma_{n+1}&=\begin{multlined}[t][0.75\textwidth]	-\gamma_{n}-(\beta_{n}+b-n)\beta_{n}+n(b-a-n+1) +\eta(\beta_{n}+a+n)+\eta^{-1}(\beta_{n-1}+\beta_{n}+b-n+1-\eta) \gamma_{n}.
\end{multlined}
\end{align*}
\end{theorem}

\begin{rem}
	In \cite{smet_vanassche} a system of Laguerre--Freud equations was presented. There are two functions $u_n,v_n$ that are linked with $\beta_n,\gamma_n$ by
\begin{align*}
	\gamma_n&=n\eta-(a-1)u_n,& 
	\beta_n&=n+a-b-1+\eta-\frac{(a-1)v_n}{\eta},
\end{align*}
and the nonlinear system, see equations (3.2) and (3.3) in \cite{smet_vanassche}, is
\begin{align*}
	(u_n+v_n)(u_{n+1}+v_n)&=\frac{a-1}{\eta^2}v_n(v_n-\eta)\Big(v_n-\eta\frac{a-b-1}{a-1}\Big),\\
	(u_n+v_n)(u_{n}+v_{n-1})&=\frac{u_n}{u_n-\frac{\eta n}{a-1}}
	(u_n+\eta)\Big(u_n+\eta\frac{a-b-1}{a-1}\Big).
\end{align*}
Observe that the first relation, (3.2)  in  \cite{smet_vanassche}, gives $\gamma_{n+1}$ as a rational function of $(\gamma_{n},\beta_n)$, a length one relation. The rational function is a cubic polynomial divided by a linear function on the recursion coefficients.
Instead, our \eqref{eq:Laguerre-Freud-Meixner1} is a length two relation, but is quadratic polynomial in the recursion coefficients, not rational and involving cubic polynomials as does (3.1). 
The second relation (3.3)  in  \cite{smet_vanassche} gives $\beta_{n+1}$ in terms of  rational function of $(\gamma_{n+1},\beta_n)$ a length one relation. Now the rational function is a cubic polynomial divided by a quadratic one. Relation \eqref{eq:Meixner_Laguerre_Freud_step1} is length one as well, but our rational function is quadratic polynomial divided by a linear one.\end{rem}
\begin{rem}
 A nice feature of the  Smet--Van Assche system is that for the particular value $a=1$ provides the explicit expression
$\beta_n=n-b+\eta$ and $\gamma_n=n\eta$.
Indeed,  we check using SageMath 9.0 that \eqref{eq:laguerre-freud-meixner-2}, \eqref{eq:laguerre-freud-meixner-3}, \eqref{eq:Laguerre-Freud-Meixner1} and \eqref{eq:meixner_laguerre_freud_larga}  are satisfied.  For $a=1$, the weight is 
$	w(z)=\dfrac{\eta^z}{(b+1)_z}$
and the  $0$-th moment is given by the Kummer function 
$\rho_0=	M(1,b+1,\eta)=\frac{b}{\eta^b} \Exp{\eta} (\Gamma(b) - \Gamma(b , \eta))$,
where $\Gamma(b,\eta)=\int_\eta^\infty t^{b-1}\Exp{-t}\d t$ is the incomplete Gamma function.
According to \cite{filipuk_vanassche1}, it corresponds to the Charlier case on the shifted lattice $\N-b$. 
\end{rem}

\section{Gauss hypergeometric weights}\label{S:Hahn}
The choice $\sigma(z)=\eta(z+a)(z+b)$ and  $\theta(z)=z(z+c)$ leads to
 the following Pearson equation 
\begin{align*}
	(k+1)(k+1+c)w(k+1)=\eta(k+a)(k+b)w(k)
\end{align*}
whose solutions are proportional to 
	$w(z)=\frac{(a)_z(b)_z}{(c+1)_z}\frac{\eta^z}{z!}$.
According to \cite{diego,diego_paco} this is the generalized  Hahn weight of type I. The first moment is
$	\rho_0={}_2F_1\left[\!\!{\begin{array}{c}a,b \\c+1\end{array}};\eta\right]$,
i.e., the Gauss hypergeometric function. For $\eta=1$, Hahn introduced these discrete orthogonal polynomials in \cite{Hahn}. The standard \emph{Hahn polynomials} considered in the literature take $a=\alpha+1$, $b=-N$ and $c=-N-1-\beta$, with $N$ a positive integer.

\begin{rem}
	From previous comments and \eqref{eq:tau_recursion} we get that in terms of the following  $\delta$-Wronskian  of the Gauss hypergeometric  function
	\begin{align*}
		\tau_n&:=\mathscr W_{n}\left({}_2F_1\left[\!\!{\begin{array}{c}a,b \\c+1\end{array}};\eta\right]\right),
	\end{align*}
	we get explicit expressions for the recursion coefficients
	\begin{align*}
		\beta_n&=\vartheta_{\eta}\log \frac{\tau_{n+1}}{\tau_n},&  \gamma_{n+1}&=\frac{\tau_{n+1}\tau_{n-1}}{\tau_n^2},& n\in\N_0.
	\end{align*}
\end{rem}

\begin{theorem}[The Gauss hypergeometric  Laguerre--Freud  structure matrix]\label{teo:Hahn}
	For a Gauss hypergeometric weight; i.e. $\sigma=\eta(z+a)(z+b)$ and $\theta=z(z+c)$, 
	we find
\begin{align}
	 \label{eq:Hahn1_p}
	p^1_{n+1}&=\begin{multlined}[t][.85\textwidth]
		(n+2)\beta_{n+2}+\beta_{n+1}+\frac{\eta-1}{\eta+1}\Big(\gamma_{n+3}+\gamma_{n+2}+\beta^2_{n+2}+\frac{(n+2)(n+1)}{2}\Big)\\+\frac{1}{\eta+1}\big(\eta ab+(\eta(a+b)-c)\beta_{n+2}+(\eta(a+b)+c)(n+2)\big),
\end{multlined}\\
\label{eq:Hahn1_pi}
	\pi^{[2]}_n&=\begin{multlined}[t][.85\textwidth]\frac{1}{1+\eta}\big((n+2)(n+1)-\eta ab-(\eta(a+b)-c)\beta_{n+2}-(\eta(a+b)+c)(n+2)\big)\\+\frac{1-\eta}{1+\eta}(\gamma_{n+3}+\gamma_{n+2}+\beta^2_{n+2})-(n+2)(\beta_{n+2}+\beta_{n+1}).
\end{multlined}
\end{align}
The Laguerre--Freud structure matrix is
	{	\begin{align}\label{eq:Laguerre-Freud-Hann-structure}
		\hspace*{-.5cm}
		\Psi&=\left[\begin{NiceMatrix}[small,columns-width = 1cm	]
			\cellcolor{Gray!10}	\eta(\gamma_1+(\beta_0+a)(\beta_0+b) )H_0& (\beta_0+\beta_1+c)H_1& H_2& 0 &\Cdots&\\
			\eta (\beta_0+\beta_1+a+b)H_1&\cellcolor{Gray!10}\scriptsize\begin{matrix}\eta (\gamma_1+\gamma_2+(\beta_1+a)(\beta_1+b) \\+\beta_0+\beta_1+a+b)H_1
			\end{matrix}& (\beta_1+\beta_2+c-1)H_2&H_3&\Ddots&
			\\
			\eta H_2	&	\eta (\beta_1+\beta_2+a+b+1)H_2&\cellcolor{Gray!10}\scriptsize\begin{matrix}
				\eta (\gamma_2+\gamma_3+(\beta_2+a)(\beta_2+b) \\+2(\beta_1+\beta_2+a+b)+\pi^{[2]}_0)H_2
			\end{matrix}&(\beta_2+\beta_3+c-2)H_3&\Ddots&
			\\[5pt]
			0&\eta H_3 &\eta (\beta_2+\beta_3+a+b+2)H_3&\cellcolor{Gray!10}\scriptsize\begin{matrix}
				\eta (\gamma_3+\gamma_4+(\beta_3+a)(\beta_3+b) \\+3(\beta_2+\beta_3+a+b)+\pi^{[2]}_1)H_1
			\end{matrix}&\Ddots&
			\\
			\Vdots	&\Ddots[shorten-end=40pt]& \Ddots[shorten-end=90pt]&\Ddots[shorten-end=80pt]&\Ddots&
		\end{NiceMatrix}\right].
	\end{align}
}
\end{theorem}

\begin{proof}
In this case, since the polynomials $\sigma$ y $\theta$, the Freud--Laguerre matrix has the following diagonal structure 
$\Psi=(\Lambda^{\top})^2\psi^{(-2)}+\Lambda^{\top}\psi^{(-1)}+\psi^{(0)}+\psi^{(1)}\Lambda+\psi^{(2)}\Lambda^2$.
That is, it has two subdiagonals, two superdiagonals and the main diagonal.

In the one hand, from  $\Psi=\sigma(J)H\Pi^{\top}$, we get for the Laguerre-Freud structure matrix
\begin{multline}\label{eq:Psi_Hahn_1}
	\Psi=\underset{\text{second subdiagonal}}{\underbrace{\eta\Lambda^{\top}\gamma\Lambda^{\top}\gamma H}}+\underset{\text{first subdiagonal}}{\underbrace{\eta(\Lambda^{\top}\gamma\Lambda^{\top}\gamma HD\Lambda+\Lambda^{\top}\gamma(\beta +b)H+(\beta +a)\Lambda^{\top}\gamma H)}}\\+\underset{\text{main diagonal}}{\underbrace{\eta\big(\Lambda^{\top}\gamma\Lambda H+\Lambda\Lambda^{\top}\gamma H+(\beta +a)(\beta +b)H+(\Lambda^{\top}\gamma(\beta +b)+(\beta +a)\Lambda^{\top}\gamma)HD\Lambda+\Lambda^{\top}\gamma\Lambda^{\top}\gamma H\pi^{[2]}\Lambda^2\big)}}\\+\begin{multlined}[t][.8\textwidth]
\underset{\text{first superdiagonal}}{\underbrace{\eta	\Big(\big((\beta +a)(\beta +b)+\Lambda^{\top}\gamma\Lambda+\Lambda\Lambda^{\top}\gamma\big)HD\Lambda+(\beta +a)\Lambda H+\Lambda(\beta +b)H}}\\\underset{\text{first superdiagonal}}{\underbrace{+\Lambda^{\top}(\beta +b)H\pi^{[2]}\Lambda^2+(\beta +a)\Lambda^{\top}H\pi^{[2]}\Lambda^2+\Lambda^{\top}\gamma\Lambda^{\top}\gamma H\pi^{[3]}\Lambda^3\Big)
}}	\end{multlined}\\+\begin{multlined}[t][.7\textwidth]
\underset{\text{second superdiagonal}}{\underbrace{	\eta\Big(\Lambda^{\top}\gamma\Lambda^{\top}\gamma H\pi^{[4]}\Lambda^4+\big(\Lambda^{\top}\gamma(\beta +b)+(\beta +a)\Lambda^{\top}\gamma\big)H\pi^{[3]}\Lambda^3}}\\\underset{\text{second superdiagonal}}{\underbrace{+\big(\Lambda^{\top}\gamma\Lambda+\Lambda\Lambda^{\top}\gamma+(\beta +a)(\beta +b)\big)H\pi^{[2]}\Lambda^2+\big(\Lambda(\beta +b)+(\beta +a)\Lambda\big)HD\Lambda+\Lambda^2H\Big)}}.
\end{multlined}
\end{multline}
In the other hand, from  $\Psi=\Pi^{-1}H\theta(J^{\top})$ we deduce
\begin{multline}\label{eq:Psi_Hahn_2}
	\Psi=\begin{multlined}[t][.8\textwidth]
	\underset{\text{second subdiagonal}}{\underbrace{H(\Lambda^\top)^2+\Lambda^\top DH\Lambda^\top(\beta+T_{-}\beta+c)+(\Lambda^\top)^2\pi^{[-2]}H(\gamma+T_{+}\gamma+\beta^2+c\beta)}}\\\underset{\text{second subdiagonal}}{\underbrace{-(\Lambda^\top)^3\pi^{[-3]}H(\beta+T_{-}\beta+c)\gamma\Lambda+(\Lambda^\top)^4\pi^{[-4]}H(T_{-}\gamma)\gamma\Lambda^2}}
	\end{multlined}\\+\underset{\text{first subdiagonal}}{\underbrace{H\Lambda^\top(\beta+T_{-}\beta+c)-\Lambda^\top DH(\gamma+T_{+}\gamma+\beta^2+c\beta)+(\Lambda^\top)^2\pi^{[-2]}H(\beta+T_{-}\beta+c)
\gamma\Lambda-(\Lambda^\top)^3\pi^{[-3]}H(T_{-}\gamma)\gamma\Lambda^2}}\\
+\underset{\text{main diagonal}}{\underbrace{H(\gamma+T_{+}\gamma+\beta^2+c\beta)-\Lambda^\top DH(\beta+T_{-}\beta+c)\gamma\Lambda+(\Lambda^\top)^2\pi^{[-2]}H(T_{-}\gamma)\gamma\Lambda^2}}\\
+\underset{\text{first superdiagonal}}{\underbrace{H(\beta+T_{-}\beta+c)\gamma\Lambda-\Lambda^\top DH(T_{-}\gamma)\gamma\Lambda^2  }}+\underset{\text{second superdiagonal}}{\underbrace{H(T_{-}\gamma)\gamma\Lambda^2}}
\end{multline}

From \eqref{eq:Psi_Hahn_1} we get the first two subdiagonals of $\Psi$, namely
\begin{align*}
	\psi^{(-2)}&=\eta  H\gamma T_-\gamma, & \psi^{(-1)}&=\eta\gamma T_+\gamma T_+H+\gamma(\beta+b)H+(T_-\beta+a)\gamma H,
\end{align*}
and from \eqref{eq:Psi_Hahn_2} we get the first two superdiagonals of $\Psi$, namely
\begin{align*}
	\psi^{(1)}&= (\beta+T_{-}\beta+c)\gamma-( T_+D)(T_+H)\gamma T_+\gamma, & \psi^{(2)}&=H(T_{-}\gamma)\gamma.
\end{align*}

 We can obtain an expression for the main diagonal that does not depend on $\pi^{[+2]}$ o de $\pi^{[-2]}$ by equating the terms corresponding to the main diagonal in both previous expressions \eqref{eq:Psi_Hahn_1} and \eqref{eq:Psi_Hahn_2}, and we obtain
\begin{multline*}
	(\eta-1)(T_{+}\gamma+\gamma+\beta^2)+\beta[\eta(a+b)-c]+\eta ab+(\eta+1)(T_{+}D)[T_{+}\beta+\beta]+(T_{+}D)[\eta(a+b)+c]\\=T_{+}^2(\pi^{[-2]}-\eta\pi^{[2]}).
\end{multline*}

Recalling   \eqref{eq:pis}, i.e. $\pi^{[\pm 2]}_{n}=\frac{(n+2)(n+1)}{2}\mp(n+1)\beta_{n+1}\mp p_{n+1}^1$, and substituting in the previous expression we can clear $p^1_{n+1}$, and subsequently replaced it in the expression of $\pi^{[2]}$ so that the main diagonal no longer depends on it.
 Reading element by element the first equality and   clearing for  $p^1_{n+1}$,  we obtain 
 \begin{multline*}
	p^1_{n+1}=(n+2)\beta_{n+2}+\beta_{n+1}+\frac{\eta-1}{\eta+1}\Big(\gamma_{n+3}+\gamma_{n+2}+\beta^2_{n+2}+\frac{(n+2)(n+1)}{2}\Big)\\+\frac{1}{\eta+1}\big(\eta ab+(\eta(a+b)-c)\beta_{n+2}+(\eta(a+b)+c)(n+2)\big),
\end{multline*}
and the entries of the diagonal matrix $\pi^{[2]}$ are
\begin{multline*}
	\pi^{[2]}_n=\frac{1}{1+\eta}\big((n+2)(n+1)-\eta ab-(\eta(a+b)-c)\beta_{n+2}-(\eta(a+b)+c)(n+2)\big)\\+\frac{1-\eta}{1+\eta}(\gamma_{n+3}+\gamma_{n+2}+\beta^2_{n+2})-(n+2)(\beta_{n+2}+\beta_{n+1}).
\end{multline*}
Simplifying, the Laguerre--Freud matrix is gotten
\begin{multline*}
	\Psi=\eta(\Lambda^{\top})^2T_{-}^2 H+
	\eta\Lambda^{\top}\big(a+b+T_{+}D+\beta+T_{-}\beta\big)T_-H\\+\eta\big(T_{+}\gamma+\gamma+(\beta+a)(\beta+b)
	+T_{+}(D)(a+b+T_{+}\beta+\beta)+T_{+}^2\pi^{[2]})\big)H\\
+	(c-T_+D+\beta+T_{-}\beta)T_-H\Lambda+T_{-}^2H\Lambda^2.
\end{multline*}
\end{proof}

Let us analyze the compatibility $[\Psi H^{-1},J]=\Psi H^{-1}$,
\begin{theorem}[Laguerre--Freud equations for Gauss hypergeometric ]\label{teo:compatibilityI_Hahn}
The Gauss hypergeometric recursion coefficients satisfy the following Laguerre--Freud relations
\begin{subequations}
		\begin{gather}
	\label{eq:Hahn_compatibility_1}		\begin{multlined}[t][.9\textwidth]
	    (\eta^2-1)\big((\beta_{n+1}+\beta_n)\gamma_{n+1}-(\beta_{n-1}+\beta_n)\gamma_n\big)
	    \\+\eta
	    \big(\beta_n(2\beta_n+a+b+c)+2(\gamma_{n+1}+\gamma_n)+n(a+b-c+n-1)+ab\big)\\
	 +(\eta+1)\big(\big(\eta(a+b)-c-(\eta+1)n\big)(\gamma_{n+1}-\gamma_n)+(\eta+1)\gamma_n\big)=0,
	\end{multlined}\\
	\label{eq:Hahn_compatibility_2}		\begin{multlined}[t][.9\textwidth]
        (\eta+1)((n-1)\beta_n+(n+1)\beta_{n+1})+(\eta-1)(\gamma_{n+2}-\gamma_n+\beta_{n+1}^2-\beta_n^2+n)\\+(\eta(a+b)-c)(\beta_{n+1}-\beta_n)+\eta(a+b)+c=0 . 
	\end{multlined}
\end{gather}
\end{subequations}
\end{theorem}
\begin{proof}
We analyze the compatibility $[\Psi H^{-1},J]=\Psi H^{-1}$ by diagonals. In both sides of the equation we find matrices whose only non-zero diagonals are the main diagonal, the first and second subdiagonals and the first and second superdiagonals.  Equating the non-zero diagonals of both matrices, two identities for the second superdiagonal and subdiagonal are obtained. From the remaining diagonals we obtain the two Laguerre--Freud equations  (we obtain the same equality from the first subdiagonal and from the first superdiagonal). Firstly, by simplifying we obtain that: 
\begin{multline*}
	 \Psi H^{-1}=
	 \eta(\Lambda^{\top})^2(T_{-}\gamma)\gamma+\eta\Lambda^{\top}\gamma(T_{+}(D)+(\beta+b)+T_{-}(\beta+a))\\+\frac{\eta}{\eta+1}\big(2(T_+\gamma+\gamma+\beta^2)+(a+b+c)\beta+ab+T_+D(a+b-c)+T_+DT_+^2D\big)\\+(\beta+T_{-}\beta+c-T_{+}( D))\Lambda+\Lambda^2.
\end{multline*}
From the main diagonal, clearing, we obtain:
\begin{multline*}
 (1-\eta)\beta_{n-1}\gamma_n+(\eta-1)\beta_{n+1}\gamma_{n+1}+\beta_n(\frac{\eta}{\eta+1}(2\beta_n+a+b+c)+(\eta-1)(\gamma_{n+1}-\gamma_n))+\\(\frac{\eta}{\eta+1}(2(\gamma_{n+1}+\gamma_n)-nc+n(a+b+n-1)+ab)-(\eta+1)(n(\gamma_n-\gamma_{n+1})-\gamma_n)+(\eta(a+b)-c)(\gamma_{n+1}-\gamma_n))=0
\end{multline*}
and we get Equation \eqref {eq:Hahn_compatibility_1}.

From the first superdiagonal and from the first subdiagonal we get
\begin{multline*}
((\eta+1)(\beta_n+\beta_{n+1}+n(\beta_{n+1}-\beta_n))+(\eta-1)(\gamma_{n+2}-\gamma_n+\beta_{n+1}^2-\beta_n^2+n)+c(1+\beta_n-\beta_{n+1})\\+\eta(a+b)(1+\beta_{n+1}-\beta_n)=0 
\end{multline*}
that leads to \eqref{eq:Hahn_compatibility_2}.	
\end{proof}

We now proceed with the compatibility $[\Psi H^{-1},J_{-}]=\vartheta_{\eta}(\Psi H^{-1})$, recall that $J_{-}:=\Lambda^{\top}\gamma$ and  $\vartheta_{\eta}=\eta \frac{d}{d \eta}$. As we will see we get no further equations than those already obtained in Theorem \ref{teo:compatibilityI_Hahn}. 

\begin{pro}
The recursion coefficients for the Gauss hypergeometric  discrete orthogonal polynomials satisfy
\begin{subequations}
	\begin{gather}\label{eq:Hahn_compatibiliyII_1}
\vartheta_{\eta}(\beta_{n}+\beta_{n+1}+c-n)=\gamma_{n+2}-\gamma_{n},\\\label{eq:Hahn_compatibiliyII_2}
\begin{multlined}[t][0.85\textwidth]
    \vartheta_{\eta}\Bigl(\frac{\eta}{\eta+1}(2(\gamma_{n+1}+\gamma_n+\beta_n^2)+c(\beta_n-n)+n(n-1)+(a+b)(\beta_n+n)+ab)\Bigr)=\\\gamma_{n+1}(\beta_{n}+\beta_{n+1}+c-n)-\gamma_{n}(\beta_{n-1}+\beta_{n}+c-(n-1)), 
\end{multlined}\\\label{eq:Hahn_compatibiliyII_3}
\begin{multlined}[t][.9\textwidth]
   \vartheta_{\eta}(\eta\gamma_{n+1}(n+a+b+\beta_n+\beta_{n+1}))=\\\frac{\eta}{\eta+1}\gamma_{n+1}(2(\gamma_{n+2}-\gamma_n+\beta_{n+1}^2-\beta_n^2)+(a+b+c)(\beta_{n+1}-\beta_n)+2n+(a+b-c)) 
\end{multlined}\\\label{eq:Hahn_compatibiliyII_4}
\begin{multlined}[t][.9\textwidth]
    \vartheta_{\eta}(\eta\gamma_{n+1}\gamma_{n+2})=\eta\gamma_{n+1}\gamma_{n+2}(\beta_{n+2}-\beta_n+1).
\end{multlined}    
\end{gather}
\end{subequations}
\end{pro}
\begin{proof}
From the diagonals of $[\Psi H^{-1},J_{-}]=\vartheta_{\eta}(\Psi H^{-1})$ we get 
\begin{enumerate}
	\item From the first superdiagonal we obtain \eqref{eq:Hahn_compatibiliyII_1}
\item From the main diagonal cleaning up we get \eqref{eq:Hahn_compatibiliyII_2}.
\item From the first  subdiagonal we get, symplifying \eqref{eq:Hahn_compatibiliyII_3}.
\item 
 Finally, from the second subdiagonal we get \eqref{eq:Hahn_compatibiliyII_4}.
\end{enumerate}
\end{proof}
\begin{rem}
	We see that  \eqref{eq:Hahn_compatibiliyII_1} follow from the Toda equation \eqref{eq:Toda_system_beta} and \eqref{eq:Hahn_compatibiliyII_4} follow from Toda equation \eqref{eq:Toda_system_gamma}. Moreover the two remaining equations (from the main diagonal and the first subdiagonal) are those of Theorem \ref{teo:compatibilityI_Hahn}.
\end{rem}

\begin{rem}
 Dominici in  \cite[Theorem 4]{diego}  found the following Laguerre--Freud equations
	\begin{align*}
		(1-\eta)\nabla(\gamma_{n+1}+\gamma_n)&=\eta v_n\nabla(\beta_n+n)-u_n\nabla(\beta_n-n),\\
		\Delta\nabla(u_n-\eta v_n)\gamma_n&=u_n\nabla (\beta_n-n)+\nabla (\gamma_{n+1}+\gamma_n),
	\end{align*}
	with $u_n:=\beta_n+\beta_{n+1}-n+c+1$ and 	$v_n:=\beta_n+\beta_{n-1}+n-1+a+b$.
	Therefore, the first one is of type  $\gamma_{n+1} =F_1 (n,\gamma_n,\gamma_{n-1},\beta_n,\beta_{n-1})$, of length two,
	and the second of the form $\beta_{n+1 }= F_2 (n,\gamma_{n+1},\gamma_n,\gamma_{n-1},\beta_n,\beta_{n-1},\beta_{n-2})$, is of length three. 
.\end{rem}
\begin{rem}
	Filipuk and Van Assche in \cite[Equations (3.6) and (3.9)]{filipuk_vanassche2}  introduce new non local variables $(x_,y_n)$, 
	\begin{align*}
		\beta_n&=x_n+\frac{n+(n+a+b)\eta-c-1}{1-\eta},\\
	\frac{1-\eta}{\eta}\gamma_n&=y_n+\sum_{k=0}^{n-1}x_k+\frac{n(n+a+b-c-2)}{1-\eta}.	
	\end{align*}
Then, in \cite[Theorem 3.1]{filipuk_vanassche2} Equations (3.13) and (3.14) for $(x_n,y_n)$ are found, of length $0$ and $1$ respectively, in the new variables. Recall, that these new variables are non-local and involve all the previous recursion coefficients. In this respect, is it not so clear the meaning of length. The nice feature in this case, is that \cite[Equations (3.13) and (3.14)]{filipuk_vanassche2} is a discrete Painlevé equation, that combined with the Toda equations lead to a differential system for the new variables $x_n$ and $y_n$ that after suitable transformation can be reduced to Painlevé VI $\sigma$-equation. Very recently, \cite{Dzhamay_Filipuk_Stokes} it has been shown that this system is equivalent to
$dP(D^{(1)}_4/D^{(1)}_4)$, known as the difference Painlevé V.
	\end{rem}

\section*{Conclusions and outlook}
In their studies of integrable systems and orthogonal polynomials, Adler and van Moerbeke have thoroughly used the Gauss--Borel factorization of the moment matrix (see \cite{adler_moerbeke_1, adler_moerbeke_2, adler_moerbeke_4}). This strategy has been extended and applied by us in different contexts, such as CMV orthogonal polynomials, matrix orthogonal polynomials, multiple orthogonal polynomials, and multivariate orthogonal polynomials (see \cite{am, afm, nuestro0, nuestro1, nuestro2, ariznabarreta_manas0, ariznabarreta_manas2, ariznabarreta_manas_toledano}). For a general overview, see \cite{intro}.

Recently, we extended those ideas to the discrete world (see \cite{Manas_Fernandez-Irrisarri}). In particular, we applied that approach to the study of the consequences of the Pearson equation on the moment matrix and Jacobi matrices. For that description, a new banded matrix is required, the Laguerre--Freud structure matrix, which encodes the Laguerre--Freud relations for the recurrence coefficients. We have also found that the contiguous relations fulfilled by generalized hypergeometric functions, which determine the moments of the weight, describe a discrete Toda hierarchy known as the Nijhoff--Capel equation (see \cite{nijhoff}). In \cite{Manas}, we study the role of Christoffel and Geronimus transformations in the description of the mentioned contiguous relations, as well as the use of Geronimus--Christoffel transformations to characterize the shifts in the spectral independent variable of the orthogonal polynomials. 

In this paper, we delve deeper into that program and further explore the discrete semiclassical cases. We find Laguerre--Freud relations for the recursion coefficients of three types of discrete orthogonal polynomials: generalized Charlier, generalized Meixner, and Gauss hypergeometric cases. We observe that for all cases ---generalized Charlier, generalized Meixner, or Gauss hypergeometric--- a solution to each system of nonlinear equations of Laguerre--Freud type for the recursion coefficients is provided by the $\tau$-function, which is defined as a  Wronskian of the modified Bessel, Kummer, and Gauss hypergeometric functions, respectively. Notice that  in \cite{Manas_Fernandez-Irrisarri2} we presented a similar study as the one in this paper for the hypergeometric cases ${}_1 F_2, {}_2F_2$ and ${}_3F_2$.

For the future, we will extend these techniques to multiple discrete orthogonal polynomials \cite{Arvesu_Coussment_Coussment_VanAssche} and their relations with the transformations presented in \cite{bfm}, as well as quadrilateral lattices \cite{quadrilateral1, quadrilateral2}.

\section*{Acknowledgments} 
This work has its roots in several inspiring conversations with Diego Dominici during a research stay at Johannes Kepler University in Linz.

\section*{Declarations}

\begin{enumerate}
	\item \textbf{Conflict of interest:} The authors declare no conflict of interest.
	\item \textbf{Ethical approval:} Not applicable.
	\item \textbf{Contributions:} All the authors have contribute equally.
	\item \textbf{Data availability:} This paper has no associated data.
\end{enumerate}



\begin{thebibliography}{99}
	

\bibitem{adler_moerbeke_1} Mark  Adler and  Pierre van Moerbeke,
\emph{Vertex operator solutions to the discrete KP hierarchy},
Communications in Mathematical Physics \textbf{203} (1999) 185-210

\bibitem{adler_moerbeke_2} Mark  Adler and  Pierre van Moerbeke,,
\emph{Generalized orthogonal polynomials, discrete KP and Riemann–Hilbert problems,}
Communications in Mathematical Physics \textbf{207} (1999) 589-620.


\bibitem{adler_moerbeke_4}Mark  Adler and  Pierre van Moerbeke,,
\emph{Darboux transforms on band matrices, weights and associated polynomials},
International  Mathematics Research Notices \textbf{18} (2001) 935-984.

\bibitem{afm}
Carlos Álvarez-Fernández, Ulises Fidalgo Prieto, and Manuel Mañas,
\emph{Multiple orthogonal polynomials of mixed type: Gauss-Borel factorization and the multi-component 2D Toda hierarchy},
Advances in Mathematics~\textbf{227} (2011) 1451–1525.


\bibitem{am} Carlos Álvarez-Fernández and  Manuel Mañas, \emph{Orthogonal Laurent polynomials on the unit circle, extended CMV ordering and 2D Toda type integrable hierarchies}, Advances in Mathematics \textbf{240} (2013) 132-193

\bibitem{nuestro0} Carlos Álvarez-Fernández, Gerardo Ariznabarreta, Juan C. García-Ardila, Manuel Mañas,	and Francisco Marcellán, \emph{Christoffel transformations for matrix orthogonal polynomials in the real line and the non-Abelian 2D Toda lattice hierarchy}, International Mathematics Research Notices \textbf{2017} n\textsuperscript{o}5 (2017) 1285-1341.

\bibitem{nuestro1} Gerardo Ariznabarreta, Juan C. García-Ardila, Manuel Mañas, and Francisco Marcellán, 
\emph{Matrix biorthogonal polynomials on the real line: Geronimus transformations},
Bulletin of Mathematical Sciences \textbf{9} (2019) 195007 (68 pages).

\bibitem{nuestro2} Gerardo Ariznabarreta, Juan C. García-Ardila, Manuel Mañas, and Francisco Marcellán, \emph{Non-Abelian integrable hierarchies: matrix biorthogonal polynomials and perturbations} 
Journal of  Physics A: Mathematical \&  Theoretical \textbf{51} (2018) 205204.

\bibitem{ariznabarreta_manas0}
Gerardo Ariznabarreta and Manuel Mañas,
Matrix orthogonal Laurent polynomials on the unit circle and Toda type integrable systems
Advances in Mathematics \textbf{264} (2014) 396-463.

\bibitem{ariznabarreta_manas01}Gerardo Ariznabarreta and Manuel Mañas, 
\emph{Multivariate orthogonal polynomials and integrable systems},
Advances in Mathematics \textbf{302} (2016) 628–739.

\bibitem{ariznabarreta_manas2}
Gerardo Ariznabarreta and Manuel Mañas, \emph{Christoffel transformations for multivariate orthogonal polynomials},
Journal of Approximation Theory \textbf{225} (2018) 242–283.

\bibitem{ariznabarreta_manas_toledano}
Gerardo Ariznabarreta, Manuel Mañas, and Alfredo Toledano, 
\emph{CMV Biorthogonal Laurent Polynomials: Perturbations and Christoffel Formulas},
Studies in Applied Mathematics \textbf{140} (2018) 333–400.

\bibitem{Arvesu_Coussment_Coussment_VanAssche} Jorge Arvesú, Jonathan Coussement, and Walter Van Assche, \emph{Some discrete multiple orthogonal polynomials}, Journal of Computational and Applied Mathematics \textbf{153} (2003).

\bibitem{generalized_hypegeometric_functions} Richard  A. Askey and Adri B. Olde Daalhuis, \emph{Generalized hypergeometric function} (2010), in Olver, Frank W. J.; Lozier, Daniel M.; Boisvert, Ronald F.; Clark, Charles W. (eds.), NIST Handbook of Mathematical Functions, Cambridge University Press.

\bibitem{baik} Jinho Baik, Thomas Kriecherbauer, Kenneth T.-R. McLaughlin, and Peter D. Miller, \emph{Discrete Orthogonal Polynomials}, Annals of Mathematics Studies \textbf{164}, Princeton University Press, 2007.

\bibitem{Beals_Wong} Richard Beals and Roderick Wong, \emph{Special functions and orthogonal polynomials}, Cambridge Studies in Advanced Mathematics \textbf{153}, Cambridge University Press, 2016.

\bibitem{bfm}
Amílcar Branquinho, Ana Foulquié-Moreno, and Manuel Mañas, 
\emph{Multiple orthogonal polynomials: Pearson equations and Christoffel formulas}, Analysis and Mathematical Physics \textbf{12} (2022) 129. 

\bibitem{charlier} Carl  V. L. Charlier \emph{Über die Darstellung willkürlicher Funktionen}, Arkiv för Matematik, Astronomi och Fysic \textbf{2} (1905-06) 20.


\bibitem{clarkson} Peter A. Clarkson, \emph{Recurrence coefficients for discrete orthonormal polynomials and the Painlevé equations}, Journal of  Physics A: Mathematical \&  Theoretical \textbf{46} (2013) 185205.

\bibitem{clarkson2} Peter A. Clarkson, \emph{Classical solutions of the degenerate fifth Painlevé equation}, Journal of  Physics A: Mathematical \&  Theoretical \textbf{56} (2023) 134002.


\bibitem{quadrilateral1} Adam Doliwa, Paolo Maria Santini, and Manuel Mañas, \emph{Transformations of quadrilateral lattices}, Journal of Mathematical Physics \textbf{41} (2000) 944--990.


\bibitem{diego}Diego Dominici, \emph{Laguerre–Freud equations for generalized Hahn polynomials of type I}, Journal of Difference Equations and Applications \textbf{24} (2018) 916–940.

\bibitem{diego1} Diego Dominici, \emph{Matrix factorizations and orthogonal polynomials}, Random Matrices Theory Applications \textbf{9} (2020) 2040003, 33 pp.

\bibitem{diego_paco} Diego Dominici and Francisco Marcellán, \emph{Discrete semiclassical orthogonal polynomials of class one}, Pacific Journal of Mathematics \textbf{268} n\textsuperscript{o}2 (2012) 389-411.

\bibitem{diego_paco1}   Diego Dominici and Francisco Marcellán, \emph{Discrete semiclassical orthogonal polynomials of class 2} in 
\emph{Orthogonal Polynomials: Current Trends and Applications}, edited by E. Huertas and F. Marcellán,  SEMA SIMAI Springer Series, \textbf{22} (2021) 103-169, Springer.

\bibitem{Manas_Fernandez-Irrisarri2}  Itsaso Fernández-Irrisarri, and Manuel Mañas, \emph{Laguerre--Freud equations for three families of hypergeometric discrete orthogonal polynomials}, Studies in Applied Mathematics  (2023) 1-27. \hyperref{https://doi.org/10.1111/sapm.12601}{}{}{DOI:10.1111/sapm.12601}.

\bibitem{Dzhamay_Filipuk_Stokes} Anton Dzhamay, Galina Filipuk, and Alexander Stokes, \emph{Recurrence coefficients for discrete orthogonal polynomials with hypergeometric weight and discrete Painlevé equations}, Journal of Physics A: Mathematical \& Theoretical \textbf{53} (2020) 495201.

%

\bibitem{filipuk_vanassche0} Galina Filipuk and Walter Van Assche,
\emph{Recurrence coefficients of generalized Charlier polynomials and the fifth Painlevé equation},  Proceedings of  American  Mathematical  Society  \textbf{141}  (2013) 551–62.

\bibitem{filipuk_vanassche1} Galina Filipuk and Walter Van Assche,
\emph{Recurrence Coefficients of a New Generalization
	of the Meixner Polynomials}, Symmetry, Integrability and Geometry: Methods and Applications (SIGMA) \textbf{7} (2011), 068, 11 pages.

\bibitem{filipuk_vanassche2}  Galina Filipuk and Walter Van Assche,
\emph{ Discrete Orthogonal Polynomials with Hypergeometric Weights and Painlevé VI}, Symmetry, Integrability and Geometry: Methods and Applications (SIGMA) \textbf{14} (2018), 088, 19 pages.

\bibitem{freud} Géza Freud. \emph{On the coefficients in the recursion formulae of orthogonal polynomials}, Proceedings of the  Royal Irish Academy Section A \textbf{76} n\textsuperscript{o}1 (1976) 1-6.

\bibitem{Hahn} Wolfgang Hahn,\emph{Über Orthogonalpolynome, die q-Differenzengleichungen genügen}, Mathematische Nachrichten \textbf{2 }(1949) 4–34.

\bibitem{Hietarinta} Jarmo Hietarinta, Nalini Joshi and Frank W. Nijhoff, \emph{Discrete Systems and Integrabilty}, Cambridge Texts in Applied Mathematics, Cambridge University Press, 2016.

\bibitem{Ismail} Mourad E. H.Ismail, \emph{Classical and Quantum Orthogonal Polynomails in One Variable}, Encyclopedia of Mathematics and its Applications \textbf{98}, Cambridge University Press, 2009.

\bibitem{Ismail2}
Mourad E. H. Ismail and Walter Van Assche,
\emph{ Encyclopedia of Special Functions: The Askey--Bateman Project. Volume I: Univariate Orthogonal Polynomials},
Edited by Mourad Ismail, Cambridge University Press, 2020.

\bibitem{laguerre} Edmond Laguerre, \emph{Sur la réduction en fractions continues d'une fraction qui satisfait à une  équation différentialle linéaire du premier ordre dont les coefficients sont rationnels. } Journal de Mathématiques Pures et Appliquées  4\textsuperscript{e} série, tome \textbf{1} (1885) 135–165 .

\bibitem{magnus} Alphonse P. Magnus, \emph{ A proof of Freud’s conjecture about the orthogonal polynomials related to $|x|\rho\exp(-x^{2m})$, for integer $m$}, in “Orthogonal polynomials and applications (Bar-le-Duc, 1984)”,  Lecture Notes in Mathematics \textbf{1171} 362–372, Springer, 1985.

\bibitem{magnus1} Alphonse P. Magnus, \emph{On Freud’s equations for exponential weights}, Journal of Approximation Theory \textbf{46}(1) (1986) 65–99.

\bibitem{magnus2} Alphonse P. Magnus, \emph{Painlevé-type differential equations for the recurrence coefficients of semi-classical orthogonal polynomials},  Journal of Computational and Applied Mathematics \textbf{57}   (1995) 215–237.

\bibitem{magnus3} Alphonse P. Magnus, \emph{Freud’s equations for orthogonal polynomials as discrete Painlevé equations}, in “Symmetries and integrability of difference equations (Canterbury, 1996)”, London Mathematical Society Lecture Note Series \textbf{255} 228–243, Cambridge University Press, 1999.

\bibitem{intro} Manuel Mañas, \emph{Revisiting Biorthogonal Polynomials. An LU factorization discussion} in 
\emph{Orthogonal Polynomials: Current Trends and Applications}, edited by E. Huertas and F. Marcellán,  SEMA SIMAI Springer Series, \textbf{22} (2021) 273-308, Springer.

\bibitem{Manas} Manuel Mañas \emph{Pearson Equations for  Discrete Orthogonal Polynomials: III. Christoffel and Geronimus transformations}, Revista de la Real Academia de Ciencias Exactas, Físicas y Naturales. Serie A. Matemáticas  \textbf{116}, Article number: 168 (2022).

\bibitem{quadrilateral2} Manuel Mañas, Adam Doliwa, and Paolo Maria Santini, \emph{Darboux transformations for multidimensional quadrilateral lattices. I}, Physics Letters A \textbf{232} (1997) 99--105.

\bibitem{Manas_Fernandez-Irrisarri} Manuel Mañas and Itsaso Fernández-Irrisarri, \emph{Pearson Equations for  Discrete Orthogonal Polynomials: I. Generalized Hypergeometric Functions and Toda Equations}, Studies in Applied Mathematics \textbf{148} (2022) 1141-1179.

\bibitem{meixner} Josef Meixner, \emph{Orthogonale Polynomsysteme Mit Einer Besonderen Gestalt Der Erzeugenden Funktion}, Journal of the London Mathematical Society \textbf{S1–9 }(1) (1934) 6.

\bibitem{nijhoff} Frank W. Nijhoff and Hans W. Capel,\emph{ The direct linearisation approach to hierarchies of integrable PDEs in 2 + 1 dimensions: I. Lattice equations and the differential-difference hierarchies}. Inverse Problems \textbf{6} (1990) 567-590. 

\bibitem{NSU} Arthur F. Nikiforov, Sergei K. Suslov, and Vasilii B. Uvarov, \emph{Classical Orhogonal Polynomials of a Discrete Variable},  Springer Series in Computational Physics, Springer, 1991.

\bibitem{Okamoto2}  Yousuke Ohyama, Hiroyuki Kawamuko, Hidetaka Sakai and Kazuo Okamoto, \emph{Studies on the Painlevé Equations, V, Third Painlevé Equations of Special Type $\text{P}_{\text{III}}(D7)$ and $\text{P}_{\text{III}}(D8)$}, 
Journal of  Mathematical Sciences (The  University of Tokyo )\textbf{13} (2006), 145–204.

\bibitem{Okamoto} Kazuo Okamoto, \emph{Studies on the Painlevé Equations IV. Third Painleve Equation $\text{P}_{\text{III}}$},
 Funkcialaj Ekvacioj \textbf{30} (1987) 305-332.




\bibitem{smet_vanassche} Christophe Smet and Walter Van Assche, \emph{Orthogonal polynomials on a bi-lattice}, Constructive Approximation \textbf{36} (2012) 215–242.

\bibitem{walter} Walter Van Assche, \emph{Orthogonal Polynomials and Painlevé Equations}, Australian Mathematical Society Lecture Series \textbf{27}, Cambridge University Press, 2018.


\end{thebibliography}
\end{document}